\documentclass[graybox, envcountsect, envcountsame]{svmult}

\usepackage{amsfonts, amssymb, amsmath}
\usepackage{mathptmx}       
\usepackage{helvet}         
\usepackage{courier}        
\usepackage{tikz}

\usepackage{url} 

\smartqed

\tikzset{my_dot/.style={fill, circle, inner sep=0pt,minimum size=1.5pt}}
\tikzset{my_node/.style={fill, circle, inner sep=0pt,minimum size=3pt}}
\tikzset{inv/.style={fill, circle, inner sep=0pt,minimum size=0pt}}

\newtheorem{claimm}{Claim}
\newtheorem{observation}{Observation}

\newcommand \RR {\mathbb{R}}
\newcommand \ZZ {\mathbb{Z}}

\newcommand \ra {\rightarrow}

\newcommand \tr \textrm

\newcommand{\Trop}{\operatorname{Trop}}
\newcommand{\Newt}{\operatorname{Newt}}
\newcommand{\val}{\operatorname{val}}
\newcommand{\Pic}{\operatorname{Pic}}
\newcommand{\Div}{\operatorname{Div}}
\newcommand{\conv}{\operatorname{conv}}
\renewcommand{\deg}{\operatorname{val}}

\newcommand \Af {\mathbb{A}}
\newcommand \ov {\overline}
\newcommand \PP {\mathbb{P}}
\newcommand \CC {\mathbb{C}}

\newcommand \OO {\mathcal{O}}

\newcommand {\comment}[1]{}

\newcommand {\col} {\colon}

\newcommand{\Proj}{\operatorname{Proj}}
\newcommand{\mcx}{\mathfrak{X}}

\newcommand{\pui}{\mathbb{C}\{\!\{t\}\!\}}

\newcommand \TropV {\operatorname{TropV}}

\newcommand \Xan {X^{an}}

\makeindex             

\begin{document}

\title*{Theta characteristics of tropical $K_4$-curves}
\titlerunning{Theta characteristics of tropical $K_4$-curves}
\author{Melody Chan and Pakawut Jiradilok}
\authorrunning{M. Chan and P. Jiradilok}
\institute{
Melody Chan \at
Brown University Department of Mathematics, Box 1917, Providence, RI 02912, \email{mtchan@math.brown.edu} \\
Pakawut Jiradilok \at
Harvard University Department of Mathematics, One Oxford St, Cambridge, MA 02138 \email{pjiradilok@college.harvard.edu}}
%
%
\maketitle

\abstract{A $K_4$-curve is a smooth, proper curve $X$ of genus 3 over a nonarchimedean field whose Berkovich skeleton $\Gamma$ is a complete graph on 4 vertices.  The curve $X$ has 28 effective theta characteristics, i.e.~the 28 bitangents to a canonical embedding, while $\Gamma$ has exactly seven effective tropical theta characteristics, as shown by Zharkov.  We prove that the 28 effective theta characteristics of a $K_4$-curve specialize to the theta characteristics of its minimal skeleton in seven groups of four.
}

\section{Introduction}

This paper provides a rigorous link between the classical and tropical theories of theta characteristics for a special class of algebraic curves that we call $K_4$-curves.  Fix an algebraically closed field, complete with respect to a nontrivial nonarchimedean valuation.  A {\em $K_4$-curve} is an algebraic curve over $K$ whose Berkovich skeleton is a metric complete graph on 4 vertices. These curves provide a convenient window into the study of theta characteristics and their tropicalizations.  

It is well known that every smooth plane quartic has exactly 28 distinct bitangents. Abstractly, these correspond to the 28 effective theta characteristics on a genus 3 nonhyperelliptic curve under its canonical embedding. In \cite{zharkov}, building on \cite{mz}, Zharkov developed a theory of theta characteristics in tropical geometry.  In this framework, a tropical curve (i.e.~metric graph) of genus 3 has exactly seven effective theta characteristics.  Zharkov's theory, while compelling, is ``synthetic'': it predated a precise connection to classical algebraic curves, as far as we know.  The specialization theorem in tropical geometry \cite{spec} has since provided a connection, showing that theta characteristics of a curve $X$ do indeed tropicalize to theta characteristics of its skeleton $\Gamma$.  Our main theorem provides a new rigorous connection between the two theories of theta characteristics past the general setup of specialization:

\begin{theorem}\label{t:intro}
Let $X$ be a $K_4$-curve with skeleton $\Gamma$.  Then the 28 effective theta characteristics of $X$ specialize to the effective theta characteristics of $\Gamma$ in seven groups of 4.
\end{theorem}

\noindent This answers a question in \cite{bitangent} in the special case of $K_4$-curves.  It complements the purely combinatorial analysis in that paper of smooth plane tropical quartics and their bitangents, which are shown to fall in seven equivalence classes.  The techniques that we use span both abstract tropical curves (notably the nonarchimedean Slope Formula of \cite{bpr}) and embedded tropical curves (notably tropical intersection theory \cite{osserman-payne, or}), and these two realms interact in an interesting way in some of our arguments.  In addition, the key input that we use from classical algebraic geometry is a calculation of the limits of the 28 bitangents in a family of plane quartics specializing to a union of four lines, carried out by Caporaso-Sernesi \cite[\S3.4]{cs}.

In fact, in the language of classical algebraic geometry, our theorem amounts to the following.  Suppose we have a one-parameter family of plane quartics degenerating to a union of four non-concurrent lines  in $\PP^2$.  (This is exactly the data of a $K_4$-curve; see Theorem~\ref{t:k4}(1).)  
Sometimes it is the case that a bitangent contact point will specialize into one of the six nodes of the special fiber.  In that case, Theorem~\ref{t:intro} gives refined information about that specialization.  It says that after a sequence of blowups replacing the node with a chain of rational curves, such that the bitangent contact point now specializes to a smooth point of the special fiber, Zharkov's algorithm describes where on that chain the specialization occurs.  

This particular interpretation aside, it is natural to ask whether Theorem~\ref{t:intro} should be true for all curves of genus 3 as well.  The answer is yes.  We note that since this paper first appeared as a preprint, Jensen and Len proved a significant generalization of Theorem~\ref{t:intro}, showing in particular that for a curve $X$ of genus $g$, every effective theta characteristic of $\Gamma$ lifts to $2^{g-1}$ even and $2^{g-1}$ odd theta characteristics on $X$~\cite{jensen-len}; this was previously obtained in the case of hyperelliptic curves by Panizzut~\cite{panizzut}.  Our techniques are different: Jensen-Len use the Weil pairing on the Jacobian of the curve, while we use the interaction betwen abstract and embedded tropical curves, and tropical intersection theory.
We also refer the reader to the article~\cite{harris-len} which contains a higher-dimensional analogue to Theorem~\ref{t:intro}, counting tropical tritangent planes to space sextic curves, and to forthcoming work of Len-Markwig that deals with liftings of tropical bitangents to plane tropical curves.

There is some interesting combinatorics and computations specific to the case of $K_4$-curves, as shown in Sections 4 and 5.  In Theorem~\ref{t:hexagonal}, we prove a statement that is stronger than Theorem~\ref{t:intro} for a generic honeycomb curve (see Definition~\ref{d:generic}): the {\em bitangents} of such a curve tropicalize exactly in seven groups of four.  This is stronger than the claim in Theorem~\ref{t:intro} in the sense that two bitangents can have different tropicalizations in $\RR^2$ even while their contact points retract to the same place on the skeleton.  Indeed, in Section~\ref{s:compute}, we give exactly such an example, of a nongeneric honeycomb plane quartic whose bitangents tropicalize in groups of 2 and 4.  Furthermore, we compute the beginnings of the Puiseux expansions of the 28 bitangents in this example.  
This section may be of independent interest to computational algebraic geometers, because it showcases some of the difficulties of computing over the field of Puiseux series, and how we used tropical techniques to carry out a computation which a priori is not tropical at all.

\section{Preliminaries}\label{s:prelim}

In this section, we develop preliminary notions on tropical curves and tropicalization, semistable models, theta characteristics, and specialization, all of which interact in our results.  We refer the reader who is interested in more details on these topics to the survey article~\cite{baker-jensen} and the textbook~\cite{ms}.

\subsection{Graphs, metric graphs, and tropical curves}\label{ss:graphs}
All graphs in this paper are finite, connected graphs, with multiple edges and loops allowed.  We write $G=(V,E)$ for a graph with vertices $V=V(G)$ and edges $E=E(G)$. 
Recall that for each $n\ge 1$, the {\em complete graph} $K_n$ is the graph on $n$ vertices with an edge between each pair of vertices.  

A {\em metric graph} $\Gamma$ is a graph $G$ together with a length function $\ell\col E(G) \rightarrow \RR_{>0}$.  Note that $\Gamma$ determines a topological space in a natural way and we will freely identify $\Gamma$ with that space.  The {\em genus} of $\Gamma$ is $b_1(\Gamma) = |E(G)|-|V(G)|+1$.
An {\em abstract tropical curve} is a vertex-weighted metric graph, i.e.~a metric graph $(G,\ell)$ together with a weight function $w\col V(G) \rightarrow\ZZ_{\ge0}$.  A {\em stable} tropical curve is an abstract tropical curve satisfying the condition that $\deg(v) \ge 2w(v)-2$ for every $v\in V(G)$ (here $\deg(v)$ denotes the graph-theoretic valence of the vertex $v$).

There is a combinatorial theory of divisors on graphs and metric graphs that mirrors, and has a precise relationship with, the classical theory of divisors on algebraic curves.  This theory was first developed by Baker-Norine \cite{bn}. 
Here we only give a bare-bones account of the part that we need; see \cite{bn,gk,mz} for more. 

A {\em divisor} on a metric graph $\Gamma$ is an element of the free abelian group $\Div \Gamma$ generated by the points $p\in \Gamma$.  Here $\Gamma$ is just a topological space; in particular $p$ may lie in the interior of an edge.  There is a {\em degree} map $\Div \Gamma \ra \ZZ$ sending $\sum a_p\!\cdot\!p$ to $\sum a_p$.  Write $\Div^d \Gamma$ for the set of divisors of degree $d$.  We say $D=\sum a_p \cdot p$ is {\em effective} if $a_p\ge 0$ for all $p$; the effective divisors are denoted $\Div_{\ge 0} \Gamma$.
The canonical divisor on $\Gamma$ is $K_\Gamma =\sum_{p\in\Gamma} (\deg(p) -2)\cdot p$, where again $\deg(p)$ denotes the graph-theoretic valence of $p$, taking $\deg(p) = 2$ when $p$ is in the interior of an edge.

If $\Gamma$ is a metric graph all of whose edge lengths lie in some subgroup $\Lambda\subseteq \RR$, then we say $\Gamma$ is $\Lambda$-rational.  In this situation, we say that $p\in \Gamma$ is a $\Lambda$-rational point if the distance from $p$ to any (equivalently every) vertex $v\in \Gamma$ is in $\Lambda.$

A {\em rational function} on $\Gamma$ is a continuous function $f$ on $\Gamma$ 
which on each edge is piecewise-linear with integer slopes.  The divisor of $f$ is 
$$\operatorname{div} f = \sum_{p\in\Gamma} (\text{sum of outgoing slopes at }p)\cdot p.$$
Such a divisor is called {\em principal}.  We say that $D\sim D'$ are {\em linearly equivalent} divisors if $D-D'$ is principal, and we define the {\em Picard group} $\Pic \Gamma = \Div \Gamma / \sim$.   
The {\em rank} $~r(D)$ of a divisor $D$ is 
$$\max \{k \in \ZZ:\text{for every }E\in\Div^k_{\ge 0} \Gamma,~D-E\sim E'\text{ for some }E'\in \Div_{\ge 0}\Gamma\}.$$
This definition implies in particular that $r(D)=-1$ if and only if $D$ is not equivalent to any effective divisor, since the condition holds vacuously for $k=-1$.
We say that an effective divisor $E$ on a metric graph $\Gamma$ is {\em rigid} if it is the unique effective divisor in its linear equivalence class.  
The following lemma characterizes rigidity; we will use it and Corollary \ref{l:rigid}  to construct lifts of canonical divisors in Section~\ref{s:k4}.

\begin{lemma}\label{l:rigid_old}
Suppose $D$ is an effective divisor on a metric graph $\Gamma$; write $D= \sum_{i=1}^k {a_i p_i}$, for $a_i>0$.  Then $D$ is rigid if and only if for every nonempty closed subset $S\subseteq \Gamma$ with $\partial S \subseteq \{p_1,\ldots,p_k\}$, there is some $p_i\in S$ with $\operatorname{outdeg}_S(p_i) > a_i$.
(Here $\operatorname{outdeg}_S(p)$ means the number of segments at $p$ that leave $S$.)
\end{lemma}

\begin{proof}
Suppose there is some closed $S$ with $\partial S \subseteq \{p_i\}$ and such that $\text{outdeg}_S(p_i)$  $\le a_i$ for each $p_i\in S$.  Then for sufficiently small $\epsilon>0$, there is a rational function $f$ on $\Gamma$ taking on value 0 on $S$, decreasing to value $-\epsilon$ along each edge leaving $S$, and taking on value $-\epsilon$ everywhere else.  Then $D + \text{div}(f)$ is effective.

Conversely, if $D$ is not rigid, choose $f\ne 0$ such that $D+ \text{div}(f)$ is effective. Then $S=\{p\in\Gamma: f \text{ is maximized at }p\}$ has $\partial S \subseteq \{p_i\}$ and $\text{outdeg}_S(p_i) \le a_i$ for each~$i$.
\qed \end{proof}

Then we immediately get:

\begin{corollary}\label{l:rigid}
Let $D$ be an effective divisor on a metric graph $\Gamma$.
\begin{enumerate}
	\item If $D = a\!\cdot\! p$ for $a> 0$ and $\Gamma\setminus\{p\}$ is connected, then $D$ is rigid if and only if $a < \text{deg}(p)$.
	\item If $D = p_1+\cdots+p_d$ for points $p_1,\ldots,p_d$ on the interiors of $d$ distinct edges whose removal does not disconnect $\Gamma$, then $D$ is rigid.
\end{enumerate}
\end{corollary}

\begin{remark}
For readers who are familiar with the notion of {\em $q$-reducedness}, we remark that $D$ is rigid if and only if it is $q$-reduced for every $q\in \Gamma$.  The only if direction is clear by definition. Conversely, if $D$ is not rigid, choose $f\ne 0$ such that $D+\text{div}(f)$ is effective. Then $D$ is not reduced with respect to any $q$ in the boundary of $\{p\in\Gamma\colon f\text{ is minimized at }p\}$.
\end{remark}

\subsection{Semistable models}\label{subsec:models}
  Throughout this paper, $K$ denotes an algebraically closed field that is complete with respect to a nontrivial nonarchimedean valuation $$\val\col K^* \ra \RR.$$  Let $\Lambda = \val(K^*) \subseteq \RR$ denote the value group of $K$.
Let $R$ denote the valuation ring of $K$, with maximal ideal $m$, and let $k=R/m$  be the residue field. 
For convenience, write
$$S_R =R[x_0,\ldots,x_n], \qquad S_K = K[x_0,\ldots,x_n], \qquad S_k = k[x_0,\ldots,x_n].$$
For $a\in R$, we write $\ov{a}$ for the reduction of $a$, i.e.~the image of $a$ in $k$.  Similarly, we write $\ov{f} \in S_k$ for the coefficientwise reduction of a polynomial $f \in S_R$.  

  Let $X$ be a finite type scheme over $K$.  
By an {\em algebraic model} for $X$ we mean a flat and finite type scheme $\mcx$ over $R$ whose generic fiber is isomorphic to $X$.  
The next lemma reviews how to compute models for subvarieties of projective space.

\begin{lemma}\label{l:model}
Let $I \subset S_K$ be a homogeneous ideal, and let $X = \Proj S_K/I$.  Let $I_R = I \cap S_R$.  Then $\mcx=\Proj S_R/I_R$ is an algebraic model for $X$.

Furthermore, if $I$ is principally generated by a polynomial $f \in S_R$ which does not reduce to 0 in $S_k$, then 
$$I \cap S_R = f S_R.$$
That is, $I \cap S_R$ is again principally generated by $f$.
\end{lemma}

\begin{proof}
For a closed subscheme of $\Af^n_K$ with defining ideal $I\subset K[x_1,\ldots,x_n]$, the fact that $I\cap R[x_1,\ldots,x_n]$ defines a flat scheme with general fiber isomorphic to $X$ is shown in \cite[Proposition 4.4]{gubler}.  The first statement then follows by checking it locally on the opens $x_i \ne 0$ \cite[Remark 4.6]{gubler}.  
On such an affine open, the defining ideal of $X$ is $I|_{x_i=1} \subset K[x_1,\ldots,\hat{x_i},\ldots,x_n]$, for which $I|_{x_i}\cap R[x_1,\ldots,\hat{x_i},\ldots,x_n]$ is the equation of a model. But it is routine to check that  $I|_{x_i=1} \cap R[x_1,\ldots,\hat{x_i},\ldots,x_n] = I_R|_{x_i=1},$ so $\Proj S_R/I_R$ gives a model for $X$.

For the second statement, suppose $I$ is generated by $f \in S_R$ with $\ov{f} \ne 0$, and suppose $g\in S_K$ with $fg\in S_R$.  We will show that $g\in S_R$.  Suppose not, and pick some scalar $a\in K$ with positive valuation such that $ag\in S_R$ and $\ov{ag} \ne 0$.  Then $\ov{afg} \ne 0$ since $\ov{f} \ne 0$ and $\ov{ag}\ne 0$.  On the other hand, $\ov{a} = 0$, contradiction.  
\qed \end{proof}

Suppose $X/K$ is a smooth, proper curve of genus $g$.  A {\em semistable model} for $X$ is a proper model $\mcx/R$ whose special fiber $\mcx_k =\mcx \times_R k$ is a nodal curve over $k$.  If $\mcx_k$ is a stable curve, then we say $\mcx$ is a {\em stable model}.  Stable models exist by \cite{bosch-lutkebohmert} when $g\ge 2$. 

To any semistable model $\mcx/R$ for $X$ we associate a tropical curve, which we denote $\Gamma_\mcx$,  as follows.  
The vertices $v_i$ of $\Gamma_\mcx$ are in bijection with the irreducible components $C_i$ of $\mcx$.  For every node $p$ of $\mcx_k$, say lying on components $C_i$ and $C_j$, the completion of the local ring $\mathcal{O}_{\mcx,p}$ is isomorphic to $R[\![x,y]\!]/(xy-\alpha)$ for some $\alpha\in m$, and $\val(\alpha)\in \Lambda_{>0}$ is independent of all choices.  Then we put an edge $e_p$ between $v_i$ and $v_j$ of length $\val(\alpha)$.  Thus $\Gamma_\mcx$ is, by construction, a metric graph with edge lengths in the value group.

\subsection{Classical and tropical theta characteristics}  \label{ss:theta}
Suppose $X$ is a smooth, proper curve over $K$ of genus $g$.  Recall that a {\em theta characteristic} of $X$ is a divisor class $[D]\in \Pic^{g-1}(X)$ such that 
$2[D]=[K_X]$.  It is {\em effective} if $D$ is linearly equivalent to an effective divisor; and it is {\em odd} (respectively {\em even}) if $h^0(\OO_X(D))$ is odd (respectively even).  It is well-known that $X$ has 
exactly $2^{2g}$ theta characteristics, and of these, exactly $(2^g-1)2^{g-1}$ are odd \cite[\S5.1.1]{dolgachev-classical}.

Now when $X$ is a nonhyperelliptic curve of genus $3$, then the effective theta characteristics of $X$ are precisely its odd theta characteristics, of which there are 28.  These are precisely divisor classes representable as $[P+Q]$ such that $2P+2Q\sim K_X$; in other words, $P$ and $Q$ are the pairs of contact points of the 28 bitangent lines to the curve $X$ under its canonical embedding as a smooth plane quartic in $\PP^2$.

Now suppose $\Gamma$ is a metric graph of genus $g$.  A {\em theta characteristic} of $\Gamma$ is a divisor class $[D]\in \Pic^{g-1}\Gamma$ such that $2[D]=[K_\Gamma]$.  Zharkov \cite{zharkov} gives an algorithm for computing theta characteristics of $\Gamma$ and shows that $\Gamma$ has exactly $2^g$ theta characteristics, all but one of which are effective. 
His description of the effective theta characteristics can be reformulated as follows.
The effective theta characteristics of $\Gamma$ are in bijection with the $2^g-1$ nonempty Eulerian subgraphs of $\Gamma$, that is, the subgraphs of $\Gamma$ that have everywhere even valence.  Namely, if $S\subseteq \Gamma$ is an Eulerian subgraph, then the distance function $d(S,-)$ produces an orientation of $\Gamma\setminus S$ in which segments are directed away from $S$. This partial orientation, taken together with a cyclic orientation on $S$, orients $\Gamma$ in its entirety.  Then the divisor $$\sum_{p\in\Gamma} (\val_+(p)-1)\cdot p$$
represents an effective theta characteristic, where $\val_+(p)$ denotes indegree, and these are shown to be pairwise linearly inequivalent.

\subsection{Specialization}  \label{ss:spec}
The theory of specialization of divisors from curves to graphs was developed in \cite{spec}; we recall the relevant facts here, starting with skeletons of Berkovich curves \cite{berk1, ber99, thuillierthesis}.

Let $X/K$ be a smooth, proper curve of genus $g\ge 2$.  
Fix a semistable model $\mcx$ for $X$ and write $\Gamma_\mcx$ for the associated metric graph, as defined in \S\ref{subsec:models}.  There is a canonical embedding of $\Gamma_\mcx$ in the Berkovich analytification $\Xan$ of $X$, with the property that $\Xan$ admits a retraction
$$\tau\col \Xan \ra \Gamma_\mcx.$$
We call $\Gamma_\mcx$ a {\em skeleton} for $X$.  If $g\ge 2$ then $X$ has a {\em minimal skeleton} $\Gamma_\mcx$, corresponding to a stable model $\mcx$ for $X$ \cite[\S4.16]{bpr}.  

Write $\Gamma=\Gamma_\mcx$ when the dependence on the model is clear.  Now, there is a natural inclusion $X(K) \hookrightarrow X^{an}$, and the composition $X(K)\hookrightarrow X^ {an} \twoheadrightarrow \Gamma$ induces, by linearity, a {\em specialization map}
$$\tau_*\col \Div(X) \rightarrow \Div (\Gamma).$$
The specialization map is, by construction, a degree-preserving group homomorphism that sends effective divisors to effective divisors.  It also sends principal divisors to principal divisors \cite{thuillierthesis} (see also \cite[Theorem 5.15]{bpr}) and so descends to a map $\Pic(X) \rightarrow \Pic(\Gamma)$ that we will also denote $\tau_*$ by slight abuse of notation.  
\begin{observation}
The specialization map $\tau_*$ takes effective theta characteristics of $X$ to effective theta characteristics of $\Gamma.$  
\end{observation}

\begin{proof}
This follows from the fact that the canonical class on $X$ specializes to the canonical class on $\Gamma$, which is proved in \cite[Lemma 4.19]{spec}.
\qed \end{proof}

\subsection{Tropicalizations}  

Suppose $V\subseteq (K^*)^n$ is a subvariety of an algebraic torus.  The {\em tropicalization} of $V$ is the closure (in the usual Euclidean topology on $\RR^n$) of the set
$$\{(\val(x_1), \ldots, \val(x_n))~|~(x_1,\ldots,x_n)\in V\} \subseteq \RR^n;$$
this set can be equipped with the structure of a polyhedral complex and positive integer multiplicities on top-dimensional faces so that the result is a balanced complex of dimension equal to the dimension of $V$.  See \cite{ms} for details.  

Now if $V\subseteq \PP^n$, then we may consider the torus $T=\{(x_0:\ldots:x_{n-1}:1)\} \subset\PP^n$; then by abuse of notation, when we refer to the tropicalization of $V$, we shall simply mean the tropicalization of $V\cap T$, i.e.~the part inside the torus.  

For example, consider a line $Ax+By+z = 0$ in $\PP^2$, with $A,B\in K^*$.  The  tropicalization of this line is the 1-dimensional polyhedral complex consisting of a point at $(-\val(A),-\val(B)) \in \RR^2$, and three rays from that point in directions $(1,0)$, $(0,1)$, and $(-1,-1)$.  In this situation we say that $(-\val(A),-\val(B))$ is the {\em center} of the tropical line.

\section{Plane quartics in $K_4$-form}\label{s:k4}

Let $X=V(f) \subset\PP^2$ be a smooth quartic curve over $K$.  After scaling, we may assume that the minimum valuation of the 15 coefficients of $f$ is zero.  Then by Lemma~\ref{l:model}, the quartic polynomial $f \in R[x,y,z]$ defines an algebraic model for $X$ that we will denote $\mcx.$  Now we give the main definition of the section.

\begin{definition}\label{d:k4}
We say that a smooth quartic $X = V(f)$ is in {\em $K_4$-form} if the special fiber of the model $\mcx$ for $X$ defined above is $$xyz(x+y+z)=0.$$ 
\end{definition}

The following theorem characterizes the curves that have embeddings in $K_4$ form.
\begin{theorem}\label{t:k4}
Let $X/K$ be a smooth, proper curve of genus 3.   The following are equivalent:

\begin{enumerate}
\item $X$ has an embedding as a smooth plane quartic in $K_4$-form. \label{it:form}
\item $X$ has an embedding as a smooth plane quartic whose tropicalization in $\RR^2$ has a strong deformation retract to a metric $K_4$. \label{it:def}
\item The minimal skeleton of $X$ is a metric $K_4$. \label{it:abs}
\end{enumerate}
\noindent If these equivalent conditions hold, we will say that $X$ is a {\em $K_4$-curve}.
\end{theorem}

We emphasize the differences between these three statements: statement (1) is an assertion about the existence of a certain algebraic model for $X$, as in Definition~\ref{d:k4}.  Statement (2) is an assertion about the existence of a tropicalization satisfying a topological criterion.  Finally, statement (3) is an intrinsic statement about $X$: it is an assertion about the topology of its minimal Berkovich skeleton.

\begin{proof}[Proof of Theorem~\ref{t:k4}]

\eqref{it:def}$\Rightarrow$\eqref{it:form}: Suppose $X$ has a plane embedding as a smooth plane quartic $C = V(f) \subset \PP^2_K$ whose tropicalization $\Trop(C)$ strongly deformation retracts to a $K_4$.  Then $\RR^2\setminus \Trop(C)$ has exactly three bounded regions, and each pair of bounded regions are separated by an edge of $\Trop(C)$.  

Let $\tilde{f} = f|_{z=1}$ be the dehomogenization, and write $\tilde{f} = \sum c_{ij} x^i y^j$.  
Let $\Newt(\tilde{f})$ denote the Newton polygon of $\tilde{f}$, i.e.~the convex hull of the lattice points $\{(i,j)\in\ZZ^2: c_{ij}\ne 0\}$.  So $\Newt(\tilde{f})$ is a subpolytope of the triangle $\conv \{(0,0),(4,0),(0,4)\}$.  

Consider the {\em Newton subdivision} $\Delta(\tilde{f})$ associated to $\tilde{f}$, by which we mean the subdivision of $\Newt(\tilde{f})$ obtained by projecting to $\RR ^2$ the lower faces of the polytope
$$P(\tilde{f}) = \conv \{(i,j, \val(c_{ij}))\} \subset \RR^3.$$  Then 
$\Delta(\tilde{f})$ is a polyhedral complex that is dual to $\TropV(\tilde{f})$ \cite[Proposition 3.1.6]{ms}.  In particular, the
three bounded regions of $\RR^2 \setminus \Trop(C)$ are in correspondence with interior vertices of $\Newt(\tilde{f})$, so these interior vertices are necessarily 
$(1,1),(2,1),$ and $(1,2)$. 

Furthermore, there are three edges of $\Trop(C)$ separating each pair of bounded regions of $\RR^2 \setminus \Trop(C)$.  Therefore, there must be an edge joining each pair of interior vertices in $\Delta(\tilde{f})$.  This shows that the triangle
\begin{equation}\label{eq:middletriangle}
T = \conv \{ (1,1),(2,1),(1,2) \}
\end{equation}
is a face of $\Delta(\tilde{f})$.  
Then, after applying a suitable projective transformation to $f$, we may assume 
that the triangle $T$ is the unique lowest face of $P(\tilde{f})$.  In other words, 
we may assume that all coefficients of $f$ have nonnegative valuation, and that the 
coefficients of valuation 0 are precisely $c_{11}, c_{21},$ and $c_{12}$.  It 
follows from Lemma~\ref{l:model} that $f$ defines a model for $X$ whose 
special fiber is of the form $ax^2 y z + bxy^2 z + c x y z^2 = xyz(ax+by+cz)=0$, 
for scalars $a,b,c\in k^*$.  Rescaling, we may assume $a=b=c=1.$

\eqref{it:form}$\Rightarrow$\eqref{it:abs}: 
A plane embedding of $X$ in $K_4$-form gives, by definition, a stable model $\mcx$ of $X$ whose skeleton is a metric $K_4$.

\eqref{it:abs}$\Rightarrow$\eqref{it:def}:  Let $\mcx$ be a stable model for $X$; then $\mcx_k$ has irreducible components $C_1,C_2,C_3,$ and $C_4$, each isomorphic to $\PP^1_k$, and intersecting pairwise.  
Let $\Gamma$ denote the minimal skeleton of $X$.
Let $\Omega^1_{\mcx/R}$ be the pushforward to $\mcx$ of the sheaf of relative K\"ahler differentials on the smooth locus of $\mcx$ over $R$. Then by \cite[Lemma 2.11]{krzb}, $\Omega^1_{\mcx/R}$ is a line bundle whose restriction to $X$ is the canonical bundle, while on $\mcx_k$ it restricts to the relative dualizing sheaf.  
Now, $X$ is not hyperelliptic, since for example $\Gamma$ is not hyperelliptic \cite{cha2} whereas hyperellipticity is preserved under passing to the skeleton \cite{spec}.  So the general fiber of $\mcx$ is canonically embedded as a smooth plane quartic over $K$, and the special fiber is also embedded as a stable curve in $\PP^2_k$.  The restriction of $\Omega^1_{\mcx/R}$ to each $C_i$ is isomorphic to $\omega_{C_i}(p^i_{1} + p^i_{2} + p^i_{3}) \cong \mathcal{O}_{C_i}(1)$, where $p^i_{1} , p^i_{2} ,$ and $p^i_{3}$ denote the nodes on $C_i$.  So $\mcx_k$ is the union of four lines in $\PP^2_k$, no three of which are concurrent.

All such quadruples of lines are projectively equivalent, so after a change of coordinates, we may assume they are the lines $$x =0,~y=0,~z=0,\quad \text{ and }\quad x+y+z=0.$$  Thus we may assume that $\mcx$ is defined by a homogeneous quartic polynomial $f\in R[x,y,z]$, and $\mcx_k = V(x^2yz + xy^2z + xyz^2)$.  From this it follows that the triangle $T$ defined in~\eqref{eq:middletriangle} is a face of the Newton subdivision $\Delta(\tilde{f})$, where $\tilde{f} = f|_{z=1}$ is the dehomogenization of $f$ and $\Delta(\tilde{f})$ is the subdivision of its Newton polygon $\Newt(f)$ defined above.

Let $v$ denote the vertex of $\Trop V(f)$ dual to $T$.  Now, $V(f)$ is smooth, so in particular it does not contain the coordinate lines $x=0, y=0,$ or $z=0$ as a component.  It follows that $\Newt(\tilde{f})$ meets the lines $i=0, j=0,$ and $i+j=4$.  
Next, $f$ has at least one of the three terms $x^4, x^3y,$ and $x^3z$ in its support, for otherwise $V(f)$ would be singular at $(1:0:0)$.  Thus 
$\Newt(\tilde{f})$ contains at least one of the points $(4,0), (3,1)$, or $(3,0)$.  
Similarly, $\Newt(\tilde{f})$ contains at least one of the points $(0,4), (1,3)$ or 
$(0,3)$, and at least one of the points $(0,0), (1,0)$, or $(0,1)$.  

Then $T$ lies in the interior of $\Newt(\tilde{f})$.  
It follows that when $v$ is deleted from $\Trop V(f)$, a cycle remains.  Furthermore, $v$ must be attached to that cycle along the three edges of $\Trop V(f)$ dual to the three edges of $T$.  Thus $\Trop V(f)$ contains a metric $K_4$ inside its bounded subcomplex; and it clearly has the homotopy type of a $K_4$, since the rank of $H_1(\Trop V(f), \ZZ)$ cannot exceed the number of interior lattice points of $\Newt(\tilde{f})$, which is exactly 3.
\qed \end{proof}

\begin{remark}
Theorem~\ref{t:k4} is quite special to the case of the complete graph on four vertices.  For example, not all genus 3 metric graphs can even be realized as the subcomplex of the dual complex to a Newton subdivision of a quartic, as in \cite[Proposition 2.3]{bitangent}.

We also emphasize:
we are not claiming that the embedding of $X \cong C \subset \PP^2_K$ may be chosen so that $\Trop(C)$ is tropically smooth.  In fact, \cite{bjms} gives 
inequalities on the edge lengths of a metric $K_4$ that are necessary and sufficient for it to be embeddable in $\RR^2$ as part of a tropically smooth plane quartic.
\end{remark}

\begin{remark}
In light of Theorem~\ref{t:k4}, we pose the following algorithmic question. Suppose $f=\sum c_{ijk}x^iy^jz^k$ is a smooth plane quartic.  How can we tell whether $f$ defines a $K_4$-curve; and if it does, how can one read off its six edge lengths? In principle, and in full generality, one could attempt to compute a semistable model and local equations for the nodes in the special fiber, as in \S\ref{ss:spec}.  But Theorem~\ref{t:k4} opens up the possibility of finding a more explicit algorithm, using tropical techniques,  in the special case of $K_4$-curves.  Indeed, the theorem shows that being a $K_4$-curve is equivalent to having a projective reembedding in $\PP^2_K$ whose Newton subdivision contains the triangle $T$ in~\eqref{eq:middletriangle}, and this property may be encoded explicitly as a system of inequalities on the valuations of the 15 coefficients defining a quartic curve.

The general algorithmic question of computing the abstract tropical curve, i.e.~minimal Berkovich skeleton, associated to a nonarchimedean curve is an interesting one.  See \cite{bolognese-brandt-chua} for more on the status of this problem as well as its relationship with computing tropical Jacobians.
\end{remark}

\begin{proposition}\label{p:groupsof4}
Suppose $C = V(f) \subset \PP^2_K$ is a smooth quartic in $K_4$-form.  Suppose the minimum valuation of the coefficients of $f$ is 0, so that $f$ defines a model $\mathcal{C}/R$ for $C$.  
Consider the 28 bitangents $l_1,\ldots,l_{28}$ of $C$, and let $L_1,\ldots,L_{28}$ 
$\subset \RR^2 $ denote their tropicalizations.  Let $P_1,\ldots,P_{28}\in \RR^2$ be the centers of the 28 tropical lines $L_i$.  Then 
\begin{enumerate}
\item Four of the $P_i$'s lie in the region $\{(a,b): a>0, a>b\}$.
\item Four of the $P_i$'s lie in the region $\{(a,b):b>0, b>a\}$.
\item Four of the $P_i$'s lie in the region $\{(a,b):a<0, b<0\}.$
\item Four of the $P_i$'s are $(0,0)$.
\item Four of the $P_i$'s lie in the region $\{(a,a):a >0\}$.
\item Four of the $P_i$'s lie in the region $\{(0,-a):a >0\}$.
\item Four of the $P_i$'s lie in the region $\{(-a,0):a >0\}$.
\end{enumerate}
\end{proposition}

Figure~\ref{fig:sectors} illustrates the regions of $\RR^2$ corresponding to the seven cases of Proposition~\ref{p:groupsof4}.

\begin{figure}
\begin{center}
\scalebox{1.25}{
\begin{tikzpicture}[my_node/.style={fill, circle, inner sep=1.75pt}, scale=.9]
\begin{scope}
\fill[gray!30] (-2,-2)--(-2,1.5)--(1.5,1.5) -- (1.5,-2)--cycle;
\node[my_node, label=right:{$\scriptstyle (0,0)$}] (V) at (0,0){};
\node[inv, label=45:$\scriptstyle y+z$] (W) at (-2,0){};
\node[inv] (NE) at (1.5,1.5){};
\node[inv, label=185:$\scriptstyle x+y$] at (1.25,1.5){};
\node[inv, label=100:$\scriptstyle x+z$](S) at (0,-2){};
\draw[->, ultra thick] (V) --  (W);
\draw[->, ultra thick] (V) --  (NE);
\draw[->, ultra thick] (V) --  (S);
\node[inv, label=190:$\scriptstyle x+y+z$] at (.1,0){}; 
\node[inv, label=right:$\scriptstyle x$] at (.5,-.75){};
\node[inv, label=135:$\scriptstyle y$] at (-1,.75){};
\node[inv, label=225:$\scriptstyle z$] at (-1,-.75){};
\end{scope}
\end{tikzpicture}
}
\caption{The seven regions of $\RR^2$ in Proposition~\ref{p:groupsof4}; Each region supports four centers of the 28 tropicalized bitangents, and their limiting equations are as shown.} \label{fig:sectors}
\end{center}
\end{figure}
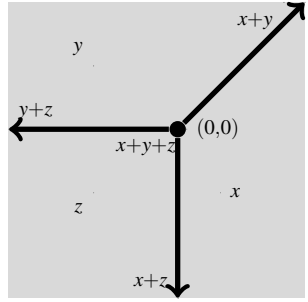

\begin{proof}
Write $l_i = V(\alpha_i x + \beta_i y + \gamma_i z)$ with $\alpha_i,\beta_i,\gamma_i \in R$ and such that the minimum valuation of $\alpha, \beta, $ and $\gamma$ is $0$.  
Then by Lemma~\ref{l:model}, the special fiber of $l_i$ is the line $\ov{\alpha_i}x + \ov{\beta_i}y + \ov{\gamma_i}z = 0 \subset \PP^2_k$.  Now, as shown by Caporaso-Sernesi \cite[3.4.11 and Lemma 2.3.1]{cs}, the 28 bitangents $l_i$ limit to the seven lines 
$$x,\quad y,\quad z,\quad x+y+z,\quad  x+y,\quad  x+z,\quad  y+z,$$ each with multiplicity four.  
  By this we mean that the closure over the space of all plane quartics of the incidence variety of bitangents over smooth quartics is flat and the fiber over the singular quartic $xyz(x+y+z)$ is as claimed.
(The lines $x+y, x+z,$ and $y+z$ are simply the three additional lines that are spanned by the six points of pairwise intersection of $x,y,z,$ and $x+y+z$.)

Consider the four lines $l_i=V(\alpha_i x + \beta_i y + \gamma_i z)$ with limit $x=0$. We have $\val(\alpha_i) = 0$ and $\val(\beta_i), \val(\gamma_i) > 0$.  Now the center $P_i$ of the tropicalized line $L_i$ is at $(-\val(\alpha_i/\gamma_i), -\val(\beta_i/\gamma_i)) \in \RR^2$, so it follows that $P_i = (a,b)$ where $a>0$ and $a>b$.  The other six cases of Proposition~\ref{p:groupsof4} can be argued analogously. 
\qed \end{proof}

Now we prove the main theorem of the section.
\begin{theorem}\label{t:groupsof4}
Let $X$ be a genus 3 smooth, proper curve over $K$ whose minimal skeleton $\Gamma$ is a metric $K_4$.  The 28 odd theta characteristics of $X$ are sent to the seven odd theta characteristics of $\Gamma$ in groups of four.
\end{theorem}

\begin{proof}
Given an effective theta characteristic $[P+Q]$ on $\Gamma$, we will show that at least four effective theta characteristics on $X$ specialize to it.  Since $X$ has 28 effective theta characteristics, it follows immediately that {\em exactly} four of them specialize to each effective theta characteristic on $\Gamma$.

We will start by using two tropical rational functions on $\Gamma$ as coordinate functions to produce an embedding of $\Gamma$ in $\RR^2$ that can be completed to a balanced tropical curve.  Then we will argue that this embedding can be lifted to a canonical embedding of $X\subset \PP^2_K$ as a plane quartic in $K_4$ form.  Finally, we will use Proposition~\ref{p:groupsof4} to find four bitangents to $X$ whose contact points with $X$ specialize to $P$ and $Q$ on $\Gamma$.  

Let $V_1,V_2,V_3,$ and $V_4$ denote the vertices of $\Gamma$, and let $E_{ij}$ denote the edge between $V_i$ and $V_j$.  Let $a,b,c,d,e,f$ denote the lengths of the edges $E_{12},E_{13},E_{14},E_{34},E_{24},E_{23}$ respectively.  Now by  \cite{zharkov} (see \S\ref{ss:theta}), the seven effective theta characteristics of $\Gamma$ are in bijection with the seven nonempty Eulerian subgraphs of $\Gamma$, namely, the four 3-cycles and the three 4-cycles in $\Gamma$.
We consider these two cases in turn.

Suppose $[P+Q]$ corresponds to a 3-cycle; after relabeling, the cycle is $V_1V_2V_3$.  After permuting $V_1,V_2,$ and $V_3$, we may assume that $c\le d,e$.  Let $x=\min(a,e,f)$ and $y=\min(b,d,f)$.  We now describe an embedding of $\Gamma$ in $\RR^2$ such that the coordinate functions are tropical rational functions on~$\Gamma,$ i.e.~they are continuous, piecewise-linear, and balanced.

Put $V_1, V_2, V_3, V_4$ at $(0,0), (x,0), (0,y), (-c,-c)$, respectively. Embed the edge $E_{14}$ in $\mathbb{R}^2$ as a straight line segment between $V_1$ and $V_4$. To embed $E_{12}$, if $a=\min(a,e,f)$, then we simply join $V_1$ and $V_2$ with a straight line segment.  Otherwise, we embed $E_{12}$ as a backtracking path $(0,0) \rightarrow (\frac{a+x}{2},0) \rightarrow (x,0)$. (Thus $E_{12}$ will eventually contribute multiplicity 2 to the segment between $(x,0)$ and $(\frac{a+x}{2},0)$ in the image of this embedding in $\RR^2$.)
Similarly, we embed $E_{13}$ by either joining directly $(0,0)$ with $(0,y)$, if $b = \min(b,d,f)$, or as a backtracking path $(0,0) \rightarrow (0, \frac{b+y}{2}) \rightarrow (0,y)$ otherwise. 

Finally, we embed the edges $E_{23}, E_{34}$, and $E_{24}$ in a piecewise-linear manner. We embed $E_{23}$  as $(x,0) \rightarrow \left(\frac{f+x}{2}, \frac{f-x}{2}\right) \rightarrow \left( \frac{f-y}{2}, \frac{f+y}{2} \right) \rightarrow (0,y)$, $E_{34}$ as $(0,y) \rightarrow \left( \frac{y-d}{2}, y\right) \rightarrow \left( \frac{-c-d}{2}, -c \right) \rightarrow (-c,-c)$, and $E_{24}$ as $(-c,-c) \rightarrow \left( -c, \frac{-c-e}{2}\right) \rightarrow \left( x, \frac{x-e}{2} \right) \rightarrow (x,0)$.  See Figure~\ref{f:embedding}.

Note that the $x$-coordinate of this embedding is a rational function $F$ on $\Gamma$ with
\begin{equation}\label{eq:divF}
\text{div}~F = 2p_z + 2q_z - 2p_x - 2q_x.
\end{equation}
The points $p_z, q_z\in \Gamma$ are the points mapping to $(\frac{-c-f}{2},-c)$ and $(-c,\frac{-c-e}{2})$.  As for $p_x$ and $q_x$, our description of where they are on $\Gamma$  depends on which of $a,e,f$ is smallest.
If $\min(a,e,f)=a$, then $p_x$ and $q_x$ are the points of $\Gamma$ mapping to $(\frac{f+a}{2}, \frac{f-a}{2})$ and $(a,\frac{a-e}{2})$.  If $\min(a,e,f)=e,$ then $p_x$ and $q_x$ map to $(\frac{f+a}{2},\frac{f-a}{2})$ and $(\frac{a+e}{2}, 0)$.  If $\min(a,e,f)=f,$ then $p_x$ and $q_x$ map to $(\frac{a+f}{2}, 0)$ and $(f,\frac{f-e}{2})$.    

Furthermore, we claim that $2p_z + 2q_z \sim K_\Gamma$.  Indeed, let $S\subset \Gamma$ be the cycle $V_1V_2V_3$ and consider the distance function $d(S,-)\col \Gamma \rightarrow \RR$ as a rational function on $\Gamma$.  Then
$$\text{div } d(S,-) = 2p_z + 2q_z - V_1 - V_2 - V_3 - V_4 = 2p_z + 2q_z - K_\Gamma.$$

Similarly, the $y$-coordinate of the embedding of $\Gamma$ described above is a rational function $G$ with
$$\text{div } G = 2p_z + 2q_z - 2p_y - 2q_y,$$ and so $2p_y + 2q_y$ is also in the canonical class of $\Gamma$.  Thus the image of $\Gamma$ under $(F,G)$ can be made into a tropical plane curve in $\RR^2$ by adding six infinite rays, each of multiplicity two.  Namely, we add two rays in direction $(1,0)$ at $p_x$ and $q_x$; two rays in direction $(0,1)$ at $p_y$ and $q_y$, and two in direction $(-1,-1)$ at $p_z$ and $q_z.$

Now, all of the edge lengths $a,b,c,d,e,f$ lie in the value group $\Lambda$, so each point $p_x,q_x,p_y,q_y,p_z,q_z$ is a $\Lambda$-rational point of $\Gamma.$  (Note that $\Lambda$ is divisible since $K$ is algebraically closed.)
Furthermore, each divisor
$$D_0 = 2p_x+2q_x \qquad D_1 = 2p_y + 2q_y \qquad D_z = 2p_z+2q_z$$
has a rigid subdivisor of degree 2.  Indeed, for $i$ being $x$, $y$, or $z$, if $p_i$ is a vertex of $\Gamma$ then $2p_i$ is rigid by Corollary~\ref{l:rigid}(1), and similarly for $q_i$.  If neither $p_i$ nor $q_i$ is a vertex of $\Gamma$, then they are on the interiors of distinct edges in $\Gamma$, and $p_i+q_i$ is again rigid by Corollary~\ref{l:rigid}(2).  Then by Corollary~\ref{c:canonical-lift}, which we postpone to the end of this section, the divisors $D_0,D_1,D_2$ can be lifted pointwise to canonical divisors $\tilde{D}_0,\tilde{D}_1,\tilde{D}_2$ on $X$.  Pick global sections $s_0,s_1,s_2\in H^0(K_X)$ with zeroes $\tilde{D}_0,\tilde{D}_1,\tilde{D}_2$ respectively.  Then the Slope Formula for tropical curves (\cite[Theorem 5.15]{bpr}, \cite[Proposition 3.3.15]{thuillierthesis}) implies that the embedding $(s_0,s_1, s_2)\col X \hookrightarrow\PP^2$ is, up to a shift in $\RR^2$,  a lift of $(F,G)\col\Gamma\rightarrow\RR^2$.  This shift can be corrected by rescaling coordinates on $\PP^2$.  Note that by Theorem~\ref{t:k4}, $X$ is canonically embedded in $\PP^2$ as a smooth plane quartic in $K_4$ form.

By Proposition~\ref{p:groupsof4}(3), there are four bitangents of $X$ whose tropicalizations $L_1,L_2,L_3,L_4$ are centered in the open southwest quadrant of $\RR^2.$  Let $C_i=(a_i,b_i)\in\RR^2$ denote the center of $L_i$.  We claim $C_i$ cannot lie to the left of the segments 
$(x,0)\rightarrow (x,\frac{x-e}{2})\ra (-c,\frac{-c-e}{2})$.  For if it did, then the rightward ray of $L_i$ would intersect $\Trop(X)$ in a point on one of these segments with stable intersection multiplicity 1, but this contradicts the main theorem of Osserman-Rabinoff \cite{or} along with the fact that each bitangent meets $X$ in two points of multiplicity 2 (or one point of multiplicity 4).  

Similarly, $C_i$ cannot lie below the segment $(0,y)\ra (\frac{y-f}{2},y) \ra (\frac{-c-f}{2}, -c).$  It follows that $a_i\le \frac{-c-f}{2}$ and $b_i \le\frac{-c-e}{2}$, and that $L_i$ intersects $\Trop(X)$ at two points, each with multiplicity 2, which retract to $p_z, q_z\in \Gamma$.  This is precisely the theta characteristic associated to the cycle $V_1V_2V_3$.  By symmetry, we have therefore proved the main claim for the four effective theta characteristics of $\Gamma$ corresponding under Zharkov's bijection to the four 3-cycles in $\Gamma$; that is, we have proved that at least four effective theta characteristics of $X$ specialize to each of these effective theta characteristic of $\Gamma$.

Finally, a similar argument using the embedding of $\Gamma$ produced above shows the claim for the three effective theta characteristics of $\Gamma$ that correspond to 4-cycles in $\Gamma$.  For example, by Proposition~\ref{p:groupsof4}(7), there are four bitangents of $X$ whose tropicalizations are tropical lines centered on the open ray $\{(\alpha,0):\alpha<0\}$.  Furthermore, it must be the case that the centers lie to the left of the segment $(\frac{y-f}{2},y)\ra (\frac{-c-f}{2},-c)$, again by \cite{or}.  This means that the two bitangent contact points tropicalize to the midpoint of $E_{34}$ and somewhere on the edge $E_{12}$; i.e.~it must be the case that those four theta characteristics on $X$  specialize to the theta characteristic associated to the cycle $V_1V_3V_2V_4$.  The other two cycles of length 4 can be argued similarly, using Propsition~\ref{p:groupsof4}(5) and (6).

\qed \end{proof}

The only part of the proof of Theorem~\ref{t:groupsof4} that was postponed is Corollary~\ref{c:canonical-lift} below, which we used to lift $(F,G)$ to a canonical embedding.  In fact, that result is an immediate consequence of the next lemma, which gives general conditions under which a divisor can be lifted pointwise.  The idea in this lemma, using rigidity to lift tropical divisors, is reminiscent of arguments that have appeared in \cite[\S 6]{jensen-payne}.

If $E$ and $E'$ are effective divisors, say that $E'$ is a {\em subdivisor} of $E$ if $E-E'$ is again effective.

\begin{lemma}\label{l:lift}
Let $X/K$ be a smooth, proper curve of genus $g\ge 2$ and let $\Gamma$ denote a skeleton for $X$.  Suppose $[\tilde{E}]\in \Pic^d(X)$ is a divisor on $X$ with $h^0(\mathcal{O}_X(\tilde{E})) \ge r+1$, and let $[E]=\tau_*([\tilde{E}])\in \Pic^d(\Gamma).$

Suppose $D\in [E]$ is an effective divisor on $\Gamma$ such that the support of $D$ is $\Lambda$-rational, and suppose that $D$ has a rigid subdivisor of degree $d-r$ (as defined in \S\ref{ss:graphs}).  Then $D$ lifts pointwise to some  $\tilde{D}\in [\tilde{E}]$.  In other words, there exists a divisor 
$\tilde{D} = \tilde{p}_1 + \cdots+\tilde{p}_d \in [\tilde{E}]$
with $\tau(\tilde{p}_i )=p_i$.  
\end{lemma}

\begin{proof}
Write $D=p_1+\cdots+p_d \in [E]$, and suppose $p_{r+1}+\cdots+p_d$ is rigid.  By \cite[\S2.3]{spec} the retraction map $\tau\col X(K)\rightarrow \Gamma$ is surjective on $\Lambda$-rational points, so we may pick arbitrary lifts $\tilde{p}_1,\ldots,\tilde{p}_r\in X(K)$ for $p_1,\ldots,p_r$.  Now since $h^0(\mathcal{O}_X(\tilde{E}))>r$, there exist $\tilde{q}_{r+1},\ldots,\tilde{q_d}$ such that 
$$\tilde{D} := \tilde{p}_1 + \cdots + \tilde{p}_r + \tilde{q}_{r+1}+\cdots+\tilde{q_{d}} \in [\tilde{E}].$$
Let $q_i = \tau(\tilde{q}_i)$.  Then $\tau_*(\tilde{D}) = p_1+\cdots+p_r + q_{r+1}+\cdots+q_d \sim D$, so 
$$q_{r+1}+\cdots+q_d \sim p_{r+1}+\cdots+p_d.$$
But $p_{r+1}+\cdots+p_d$ was assumed to be rigid, so $\tau_*(\tilde{D}) = D.$
\qed \end{proof}

\begin{corollary}\label{c:canonical-lift}
If $D\in [K_\Gamma]$ is an effective and $\Lambda$-rational divisor on a metric graph $\Gamma$ of genus $g$, and $D$ has a rigid subdivisor of degree $g-1$, then $D$ lifts pointwise to a canonical divisor on $X$.
\end{corollary}

\begin{remark}
We remark that the results of Lemma~\ref{l:lift} and Corollary~\ref{c:canonical-lift} do not follow from the main lifting criterion established in \cite[Theorem 1.1]{baker-rabinoff}. Here, our rigidity assumptions allow us to achieve a {\em pointwise lift}, whereas the lift in \cite{baker-rabinoff} may involve up to $g$ extra pairs of zeroes and poles which cancel under specialization.  Nor do our lifting results follow directly from Mikhalkin's correspondence results \cite{mikhalkin-enumeration}, since it is important for us to be able to lift the given tropical plane quartic curve {\em together with a tropical canonical embedding} of it to an algebraic curve.
\end{remark}

\begin{figure}
\begin{center}
\scalebox{1}{
\begin{tikzpicture}[my_node/.style={fill, circle, inner sep=1.75pt}, scale=.8]
\begin{scope}
\node[my_node, label=270:{$\scriptstyle (0,0)$}] (V1) at (0,0){};
\node[my_node, label=right:{$\scriptstyle (x,0)$}] (V2) at (2,0){};
\node[my_node, label=90:{$\scriptstyle (0,y)$}] (V3) at (0,1){};
\node[my_node, label=30:{$\scriptstyle (\frac{f+x}{2},\frac{f-x}{2})$}] (V23) at (2.5,0.5){};
\node[my_node, label=0:{$\scriptstyle (\frac{f-y}{2},\frac{f+y}{2})$}] (V32) at (1,2){};
\node[my_node, label=0:{$\scriptstyle (-c,-c)$}] (V4) at (-1,-1){};
\node[my_node, label=180:{$\scriptstyle (\frac{-c-d}{2},-c)$}] (V43) at (-2,-1){};
\node[my_node, label=180:{$\scriptstyle (\frac{y-d}{2},y)$}] (V34) at (-1,1){};
\node[my_node, label=270:{$\scriptstyle (x,\frac{x-e}{2})$}] (V24) at (2,-1.25){};
\node[my_node, label=0:{$\scriptstyle (-c,\frac{-c-e}{2})$}] (V42) at (-1,-2.75){};
\draw[->, ultra thick] (V23) --  (3.5,0.5);
\draw[->, ultra thick] (V32) --  (1,2.75);
\draw[->, ultra thick] (V24) --  (3.5,-1.25);
\draw[->, ultra thick] (V42) --  (-1.75,-3.5);
\draw[->, ultra thick] (V43) --  (-3,-2);
\draw[->, ultra thick] (V34) --  (-1,2.75);
\draw[-, ultra thick] (V1) --  (V2);
\draw[-, ultra thick] (V2) --  (V23);
\draw[-, ultra thick] (V23) --  (V32);
\draw[-, ultra thick] (V32) --  (V3);
\draw[-, ultra thick] (V3) --  (V1);
\draw[-, ultra thick] (V34) --  (V3);
\draw[-, ultra thick] (V4) --  (V1);
\draw[-, ultra thick] (V4) --  (V43);
\draw[-, ultra thick] (V43) --  (V34);
\draw[-, ultra thick] (V2) --  (V24);
\draw[-, ultra thick] (V24) --  (V42);
\draw[-, ultra thick] (V42) --  (V4);
\end{scope}
\begin{scope}[shift={(0,-7)}]
\node[my_node, label=270:{$\scriptstyle (0,0)$}] (V1) at (0,0){};
\node[my_node, label=90:{$\scriptstyle (x,0)$}] (V2) at (2,0){};
\node[my_node, label=270:{$\scriptstyle (\frac{a+x}{2},0)$}] (V22) at (2.5,0){};
\node[my_node, label=90:{$\scriptstyle (0,y)$}] (V3) at (0,1){};
\node[my_node, label=30:{$\scriptstyle (\frac{f+x}{2},\frac{f-x}{2})$}] (V23) at (2.5,0.5){};
\node[my_node, label=0:{$\scriptstyle (\frac{f-y}{2},\frac{f+y}{2})$}] (V32) at (1,2){};
\node[my_node, label=0:{$\scriptstyle (-c,-c)$}] (V4) at (-1,-1){};
\node[my_node, label=180:{$\scriptstyle (\frac{-c-d}{2},-c)$}] (V43) at (-2,-1){};
\node[my_node, label=180:{$\scriptstyle (\frac{y-d}{2},y)$}] (V34) at (-1,1){};
\node[my_node] (V24) at (2,0){};
\node[my_node, label=0:{$\scriptstyle (-c,\frac{-c-e}{2})$}] (V42) at (-1,-1.5){};
\draw[->, ultra thick] (V23) --  (3.5,0.5);
\draw[->, ultra thick] (V32) --  (1,2.75);
\draw[->, ultra thick] (V22) --  (3.5,0);
\draw[->, ultra thick] (V42) --  (-1.75,-2.25);
\draw[->, ultra thick] (V43) --  (-3,-2);
\draw[->, ultra thick] (V34) --  (-1,2.75);
\draw[-, ultra thick] (0,0.05) --  (2.5,0.05);
\draw[-, ultra thick] (V2) --  (V23);
\draw[-, ultra thick] (V23) --  (V32);
\draw[-, ultra thick] (V32) --  (V3);
\draw[-, ultra thick] (V3) --  (V1);
\draw[-, ultra thick] (V34) --  (V3);
\draw[-, ultra thick] (V4) --  (V1);
\draw[-, ultra thick] (V4) --  (V43);
\draw[-, ultra thick] (V43) --  (V34);
\draw[-, ultra thick] (V2) --  (V24);
\draw[-, ultra thick] (V24) --  (V42);
\draw[-, ultra thick] (V42) --  (V4);
\draw[-, ultra thick] (V2) --  (V22);
\end{scope}
\end{tikzpicture}
}
\caption{An embedding of $\Gamma$ using tropical rational coordinate functions as in the proof of Theorem~\ref{t:groupsof4}.  The six infinite rays have multiplicity 2; the bounded segments have multiplicity 1.  On the top, we depict the case $\min(a,e,f)=a$ and $\min(b,d,f)=b$.  On the bottom, we depict the case $\min(a,e,f)=e$ and $\min(b,d,f)=b$.} \label{f:embedding}
\end{center} 
\end{figure}
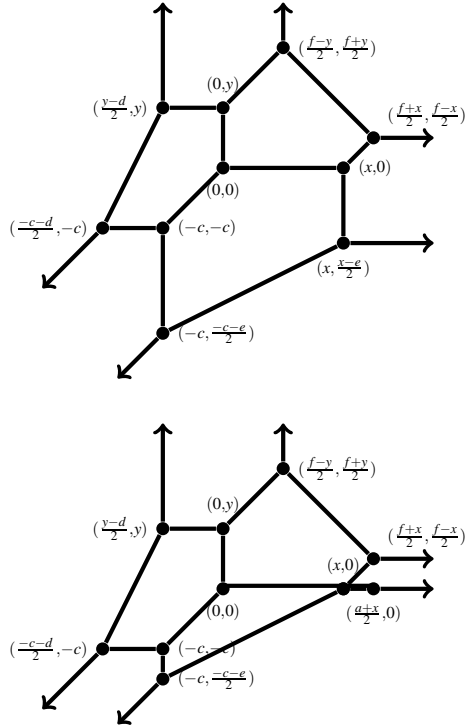

\section{The $28$ bitangents of honeycomb quartics}\label{s:honeycomb}

Let $X=V(f)\subset \PP^2_K$ be a smooth plane quartic curve.  Following \cite{knutson-tao}, we say that $X$ is a {\em honeycomb curve} if the regular subdivision of the triangle $\Delta_4 = \conv\{(0,0),(4,0),$ $(0,4)\}$ induced by $f$ is the standard one, i.e.~obtained by slicing $\Delta_4$ by the lines $x=i, y=i, x+y=i$ for $i=1,2,3$.  See Figure~\ref{fig:hexagonal}.

Honeycomb curves play a special role in tropical geometry \cite{chan-sturmfels, speyer-honeycomb}.  In this section, we will determine almost exactly where the 28 bitangents of honeycomb quartic curves go under tropicalization.  In fact, our description is exact for {\em generic} honeycomb curves, in a sense that we will define below and which is specific to this paper.  In the final section of the paper, we will give an interesting explicit calculation of the 28 bitangents of a honeycomb plane quartic over $\pui$, giving an idea of the behavior that can arise in the non-generic case.  The tools we use in this section are tropical intersection theory \cite{osserman-payne, or}, the census of bitangents and their limits provided in \cite{cs}, and the elementary geometry of tropical plane curves.

The curve $\Trop X$ divides $\RR^2$ into 15 regions, one for each lattice point in $\Delta_4$.  Let $R_{ij}$ denote the region corresponding to the lattice point $(i,j)$.  See Figure~\ref{fig:hexagonal}.  Thus there are three bounded hexagonal regions $R_{11}, R_{12}, R_{21}$. Call their unique common vertex $O$.  If $R_{ij}$ and $R_{kl}$ are adjacent regions, write $E_{ij,kl}$ for the unique edge that they share, and write $\ell({E_{ij,kl}})$ for the line spanned by $E_{ij,kl}$.

\begin{figure}
\begin{center}
\scalebox{1}{
\begin{tikzpicture}[my_node/.style={fill, circle, inner sep=1.75pt}, scale=.9]
\begin{scope}[shift={(0,-2)}]
\draw[-, ultra thick] (0,0)--(4,0)--(0,4)--(0,0);
\draw[-, ultra thick] (0,1)--(3,1)--(3,0)--(0,3)--(1,3)--(1,0)--(0,1);
\draw[-, ultra thick] (0,2)--(2,2)--(2,0)--cycle;
\end{scope}
\begin{scope}[shift={(9,0)}]
\draw[-,ultra thick, gray!30] (0,0.5)--(0,0.8)--(-.7,.8)--(-1,.5);
\draw[->,ultra thick, gray!30] (-.7,.8)--(-.7,2.5);
\draw[-,ultra thick, gray!30] (0,0.8)--(1.2,2)--(1.5,2);
\draw[->,ultra thick, gray!30] (1.2,2)--(1.2,3);
\draw[-,ultra thick ] (0, 0.5) -- (0,0) -- (2,0) -- (2.5,.5)-- (2.5,2) -- (3,2.5) -- (2.5,2) -- (1.5,2) -- cycle;
\draw[->,ultra thick] (2.5,.5) -- (3.5,.5);
\draw[->,ultra thick] (3,2.5) -- (3.5,2.5);
\draw[->,ultra thick] (3,2.5) -- (3,3);
\draw[->,ultra thick] (1.5,2) -- (1.5,3);
\draw[-,ultra thick ] (0, 0.5) -- (-1,0.5) -- (-1.5,0) -- (-1.5,-.5) -- (-.5,-.5) -- (0,0);
\draw[-,ultra thick] (-1.5,0) -- (-2.5,0);
\draw[->,ultra thick] (-2.5,0) -- (-3,-.5);
\draw[->,ultra thick] (-2.5,0) -- (-2.5,1.5);
\draw[->,ultra thick] (-1,0.5) -- (-1,2.5);
\draw[->,ultra thick] (-1.5,-0.5) -- (-2,-1);
\draw[-,ultra thick ] (-.5,-.5)--(-.5,-1) -- (1.75,-1)--(1.75, -1.5)--(1.75,-1)--(2,-.75)--(2,0);
\draw[->,ultra thick] (1.75, -1.5) -- (2.75,-1.5);
\draw[->,ultra thick] (1.75, -1.5) -- (1,-2.25);
\draw[->,ultra thick] (-.5,-1) -- (-1,-1.5);
\draw[->,ultra thick] (2,-.75) -- (3,-.75);
\node at (1.5,1){$\scriptstyle 11$};
\node at (.75,-.5){$\scriptstyle 12$};
\node at (-.75,0){$\scriptstyle 21$};
\node at (3.4,2.9) {$\scriptstyle 00$};
\node at (3.25,1.5) {$\scriptstyle 01$};
\node at (3,-.2) {$\scriptstyle 02$};
\node at (2.5, -1.1) {$\scriptstyle 03$};
\node at (2,-2) {$\scriptstyle 04$};
\node at (.6, -1.5 ) {$\scriptstyle 13$};
\node at (-1.1,-1) {$\scriptstyle 22$};
\node at (-2.1,-.5) {$\scriptstyle 31$};
\node at (-3,0.5) {$\scriptstyle 40$};
\node at (-1.75,1) {$\scriptstyle 30$};
\node at (0,1.5) {$\scriptstyle 20$};
\node at (2,2.5) {$\scriptstyle 10$};
\end{scope}
\end{tikzpicture}
}
\caption{A tropical honeycomb plane quartic $\Trop X \subset \RR^2$ and the subdivision of $\Delta_4$ to which it is dual.  The regions in the complement of $\Trop X$ are in bijection with the lattice points in $\Delta_4$ as labeled. (The curve shown in grey illustrates Remark~\ref{r:nochange}.)} \label{fig:hexagonal}
\end{center}
\end{figure}
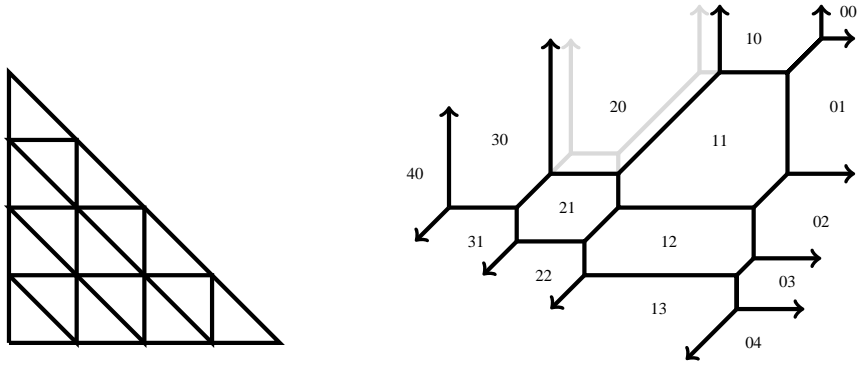

Now, the line $\ell({E_{11,12}})$ intersects the boundary of $R_{21}$ in two points: $O$ and another point, which we call $S_x$. Similarly, the line $\ell({E_{11,21}})$ intersects the boundary of $R_{12}$ at $O$ and another point $S_y$; and the line $\ell({E_{12,21}})$ intersects the boundary of $R_{11}$ at $O$ and another point $S_z$.  Finally, let $T_x = \ell({E_{01,11}} )\cap \ell( {E_{03,12}}) \in \RR^2$, let $T_y = \ell( {E_{10,11}}) \cap \ell({E_{30,21}})$, and let $T_z = \ell({E_{21,31}}) \cap \ell({E_{12,13}})$.  See Figure~\ref{fig:generic honeycomb bitangents}.

\begin{definition}\label{d:generic}
Let $X = V(f)$ be a honeycomb curve, where 
$$f(x,y,z) = 
\sum_{i+j\le 4} c_{ij} x^iy^jz^{4-i-j}.$$
Let $a_{ij} = \val ( c_{ij}).$ 
We say $X$ is a {\em generic} honeycomb curve if the following hold:
\begin{eqnarray*}
a_{31}+a_{11}-a_{30}-a_{12} &\ne & 0,\\
a_{03}+a_{21}-a_{13}-a_{11} &\ne & 0,\\
a_{10}+a_{12}-a_{01}-a_{21} &\ne & 0.
\end{eqnarray*}
\end{definition}

\begin{theorem}\label{t:hexagonal}
Let $X=V(f) \subset \PP^2_K$ be a honeycomb plane quartic curve.  Consider the 28 bitangents of $X$ and their tropicalizations $L_1,\ldots,L_{28}$.  Let $P_i\in \RR^2$ denote the center of the tropical line $L_i$.  Then:
\begin{enumerate}
	\item Four of the $P_i$'s are $T_x$,
	\item Four of the $P_i$'s are $T_y$,
	\item Four of the $P_i$'s are $T_z$,
	\item Four of the $P_i$'s are $O$,
	\item Four of the $P_i$'s lie on the closed ray in direction $(1,1)$ based at $S_z$,
	\item Four of the $P_i$'s lie on the closed ray in direction $(-1,0)$ based at $S_x$,
	\item Four of the $P_i$'s lie on the closed ray in direction $(0,-1)$ based at $S_y$.
\end{enumerate}
Moreover, if $X$ is a generic honeycomb curve, then the points in (5), (6), and (7) must be exactly $S_z,S_x,$ and $S_y$, respectively.
\end{theorem}

\begin{remark}
It is straightforward to calculate the coordinates of each of these points exactly in terms of the $a_{ij}$'s, using the duality between tropical plane curves and their lifted Newton subdivisions \cite{ms}.
Thus for almost all honeycomb quartics $X$, we can produce an explicit formula for the tropicalizations of the 28 bitangents of $X$ in terms of the coefficients of its defining equation.
\end{remark}

\begin{proof}[Proof of Theorem~\ref{t:hexagonal}]
By Proposition~\ref{p:groupsof4}, applied to $X$ under a change of coordinates if necessary so that $O = (0,0)$, each of the seven regions shown in Figure~\ref{fig:sectors} supports exactly four $P_i$'s.  Thus part (4) is immediate.  Now to prove (3), we know that four $P_i$'s lie southwest of $O$. Then each of the north and east rays of the corresponding lines $L_i$ must intersect $\Trop X$ in a single connected component of stable intersection multiplicity 2 (again, and throughout, by \cite{or}).  It follows that $P_i$ must lie on or below the closed segment $E_{21,31}$; and similarly, it must lie on or to the left of the closed segment $E_{12,13}$.  Therefore four points $P_i$ are exactly $T_z$.  A similar argument shows (1) and (2).  

\begin{figure}
\begin{center}
\scalebox{1}{
\begin{tikzpicture}[my_node/.style={fill, circle, inner sep=1.75pt}, scale=.5]
\draw[-,ultra thick ] (-1,1) -- (-1,-2) -- (-2,-3) -- (-2,-4) -- (1,-4) -- (2,-3) -- (2,-2) -- (5,1) -- (5,3) -- (1,3) -- (-1,1) -- (-2,1) -- (-4,-1) -- (-4,-3) -- (-2,-3);
\draw[-,ultra thick] (-1,-2) -- (2,-2);
\draw[-,ultra thick] (-6,-1) -- (-4,-1);
\draw[-,ultra thick] (1,-6) -- (1,-4);
\draw[-,ultra thick] (5,3) -- (6,4);
\draw[->,ultra thick] (-6,-1) -- (-6,2);
\draw[->,ultra thick] (-2,1) -- (-2,5.46);
\draw[->,ultra thick] (1,3) -- (1,6.79);
\draw[->,ultra thick] (6,4) -- (6,7);
\draw[->,ultra thick] (6,4) -- (9,4);
\draw[->,ultra thick] (5,1) -- (8.76,1);
\draw[->,ultra thick] (2,-3) -- (6.82,-3);
\draw[->,ultra thick] (1,-6) -- (4,-6);
\draw[->,ultra thick] (1,-6) -- (-1,-8);
\draw[->,ultra thick] (-2,-4) -- (-4.31,-6.31);
\draw[->,ultra thick] (-4,-3) -- (-5.74,-4.74);
\draw[->,ultra thick] (-6,-1) -- (-7,-2);
\draw[-,dashed,ultra thick,gray!70] (-2,1) -- (0,3) -- (1,3);
\draw[-,dashed,ultra thick,gray!70] (5,1) -- (5,0) -- (2,-3);
\draw[-,dashed,ultra thick,gray!70] (-2,-4) -- (-4,-4) -- (-4,-3);
\draw[-,dashed,ultra thick,gray!70] (-1,-2) -- (4,3);
\draw[-,dashed,ultra thick,gray!70] (-1,-2) -- (-1,-4);
\draw[-,dashed,ultra thick,gray!70] (-1,-2) -- (-4,-2);
\node[my_node, label=135:{$O$}] at (-1,-2){};
\node[my_node, label=180:{$S_x$}] at (-4,-2){};
\node[my_node, label=270:{$S_y$}] at (-1,-4){};
\node[my_node, label=90:{$S_z$}] at (4,3){};
\node[my_node, label=0:{$T_x$}] at (5,0){};
\node[my_node, label=90:{$T_y$}] at (0,3){};
\node[my_node, label=270:{$T_z$}] at (-4,-4){};
\end{tikzpicture}
}
\caption{The seven locations $O, T_x, T_y, T_z, S_x, S_y, S_z$, of the centers of the bitangents to a generic honeycomb quartic. Exactly $4$ bitangents are centered at each of the locations.} \label{fig:generic honeycomb bitangents}
\end{center}
\end{figure}
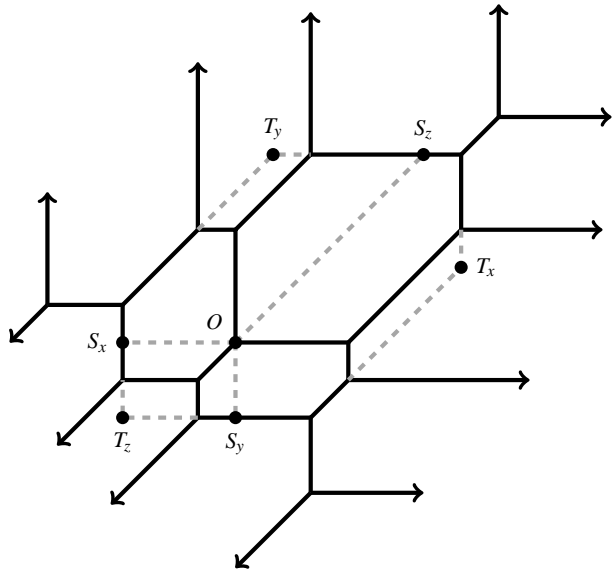

Now for (5), Proposition~\ref{p:groupsof4} guarantees that four of the $P_i$ lie in the direction $(1,1)$, i.e.~due northeast, from $O$.  If $P_i$ lies within the open region $R_{11}$, then $L_i \cap \Trop X$ would have at least three distinct connected components, contradicting that $L_i$ is a tropicalization of a bitangent.  This proves that (5) holds.  Moreover, if $X$ is generic, then the only way that the connected component of $L_i\cap \Trop X$ that contains $S_z$ can have stable intersection multiplicity 2 is if $P_i = S_z$, since otherwise $L_i\cap \Trop X$ would have more than two connected components.  This proves that four of the $P_i$s are $S_z$ when $X$ is a generic honeycomb curve.  The statements (6) and (7) are argued analogously.  
\qed \end{proof}

We note that in the case of nongeneric honeycombs, the bitangent centers need not be grouped in fours and may even appear in the regions $R_{00}, R_{40},$ and $R_{04}$, as we demonstrate in the next section; see Figure~\ref{fig:honeycomb}.

\begin{remark}\label{r:nochange}
Classically, a smooth plane quartic curve is determined uniquely by its 28 bitangents \cite{cs, lehavi}.  It is natural to ask whether the same could be true tropically.  Precisely: if $X$ and $X'\subset \PP^2$ are plane quartics with smooth tropicalization, and the 28 tropicalizations of their bitangents agree, does it follow that the tropicalizations of $X$ and $X'$ agree?  Theorem~\ref{t:hexagonal} shows that the answer is no.  In fact, it allows us to construct infinite families of tropical honeycomb quartics with the property that all lifts of them to classical curves have the same 28 tropicalized bitangents.  For example, any two curves tropicalizing to the honeycomb curve in Figure~\ref{fig:hexagonal} and the one obtained from it by shrinking region $R_{20}$ (as shown in Figure~\ref{fig:hexagonal} in grey) have the same 28 tropicalized bitangents.  This follows immediately from Theorem~\ref{t:hexagonal}, assuming that both tropical curves are generic in the sense of Definition~\ref{d:generic}.
\end{remark}

\section{Computing the $28$ bitangents of a honeycomb quartic}\label{s:compute}

In the final section of this paper, we demonstrate a computation of the Puiseux expansions of the 28 bitangents of a $K_4$-quartic, which we are able to carry out to any desired precision.  Our computation gives an example of how tropical geometry may be used in computations that are not a priori tropical.
Let $K=\pui$ in this section. Our example will be the smooth plane quartic $X$ defined by the equation
\begin{align}\label{eq:f}
\begin{split}
f(x,y,z) &=xyz(x+y+z) + t(x^2y^2 + x^2z^2 + y^2z^2) \\
&+ t^2(x^3y+xy^3+x^3z+xz^3 + y^3z+yz^3) + t^5(x^4+y^4+z^4).
\end{split}
\end{align}
The tropicalization of $X$ is a smooth tropical plane curve in honeycomb form in which every bounded segment has lattice length 1.  It is an example of a non-generic honeycomb curve, in the sense of Definition~\ref{d:generic}.  All the computations in this section were carried out in \emph{Macaulay2} \cite{m2}.  The results of our computations are summarized in Table~\ref{table:28}.

Let $\mathcal{B}$ be the set of all $(A,B) \in \pui^2$ such that $Ax+By+z=0$ is a bitangent to $f$. It is straightforward to verify that $f$ admits no bitangent of the form $Ax+By=0$. Therefore, $|\mathcal{B}| = 28$. We begin by observing that $Ax+By+z=0$ is a bitangent to $X$ if and only if the polynomial $f(x,y,-Ax-By)$ is a perfect square.  To detect this condition, we introduce a {\em square-detecting} ideal in the next lemma.

\begin{lemma} \label{p:sdi}
Let $K$ be an algebraically closed field of characteristic 0.  Let  $J \subseteq K[X_0, \dots, X_4]$ be the ideal generated by the seven cubic polynomials 
\begin{align*}
& 8X_1X_4^2-4X_2X_3X_4+X_3^3, \\
& 16X_0X_4^2+2X_1X_3X_4-4X_2^2X_4+X_2X_3^2, \\ & 8X_0X_3X_4-4X_1X_2X_4+X_1X_3^2, \\&  X_0X_3^2-X_1^2X_4, \\ & 8X_0X_1X_4-4X_0X_2X_3+X_1^2X_3, \\
& 16X_0^2X_4+2X_0X_1X_3-4X_0X_2^2+X_1^2X_2, \\ & 8X_0^2X_3-4X_0X_1X_2+X_1^3.
\end{align*}
Then  $(c_0, \dots, c_4) \in V(J)$ if and only if the polynomial 
$$s(x,y) = c_4x^4+ c_3x^3y + c_2x^2y^2 + c_1xy^3 +c_0y^4 \in K[x,y]$$ is a perfect square.
\end{lemma}

\begin{proof}
Since $K$ is algebraically closed, the polynomial $s(x,y)$ is a perfect square if and only if there exist  $C,D,C',D' \in K$ such that
\[
s(x,y) = (Cx+Dy)^2(C'x+D'y)^2.
\]
Expanding 
and eliminating $C,D,C',$ and $D'$ 
produces the ideal $J$ above.
\qed \end{proof}

Now we expand $f(x,y,-Ax-By)$ as a homogeneous quartic polynomial in $x$ and $y$ whose five coefficients are polynomials in $A$ and $B$.  Substituting these five polynomials for $X_0,\ldots,X_4$ in $J$  yields an ideal $I\subset K[A,B]$ generated by seven polynomials whose variety is $\mathcal{B}$, the bitangents of $X$.  
(The equations of $I$ are not shown here, due to the length of output.)

Our goal is to compute $V(I)$. Even though this variety is just 28 points, it is not simple to carry out its computation over the field of Puiseux series.  We now explain our strategy for carrying it out. First, we will determine the 28 {valuations} $(\val(A),\val(B))\in \RR^2$ of the bitangent coefficients, i.e.~we will compute the locations of the 28 tropicalized bitangents.  
For this we use elimination theory and Newton polygons to determine the valuations of $V(I)$ under specially chosen projections.
Second, we will bound the {denominators} of the exponents of $t$ that show up in the Puiseux expansions of pairs $(A,B)$, allowing us to pass from Puiseux series to power series (after an appropriate base change).  
Finally, we will use repeated specialization to $t=0$ to successively compute the the Puiseux expansions of the bitangent coefficients at each of the determined locations.

The first step is accomplished in the proposition below, whose proof makes use of computations in \emph{Macaulay2}.  See Figure~\ref{fig:honeycomb}.  (Note that a bitangent $(A,B)\in\mathcal{B}$ tropicalizes to a line centered at $(-\val(A),-\val(B))\in\RR^2$.)

\begin{figure}
\begin{center}
\scalebox{1}{
\begin{tikzpicture}[my_node/.style={fill, circle, inner sep=3.18pt},my_small_node/.style={fill, circle, inner sep=2.25pt}, scale=.8]
\draw[-,ultra thick] (0,1) -- (0,0) -- (1,0) -- (2,1) -- (2,2) -- (1,2) -- (0,1) -- (-1,1) -- (-2,0) -- (-2,-1) -- (-1,-1) -- (-1,-2) -- (0,-2) -- (1,-1) -- (1,0) -- (0,0) -- (-1,-1);
\draw[-,ultra thick] (-3,0) -- (-2,0);
\draw[-,ultra thick] (0,-3) -- (0,-2);
\draw[-,ultra thick] (2,2) -- (3,3);
\draw[->,ultra thick] (3,3) -- (3,4.886);
\draw[->,ultra thick] (3,3) -- (4.886,3);
\draw[->,ultra thick] (1,2) -- (1,4.571);
\draw[->,ultra thick] (2,1) -- (4.571,1);
\draw[->,ultra thick] (-1,1) -- (-1,4);
\draw[->,ultra thick] (1,-1) -- (4,-1);
\draw[->,ultra thick] (-3,0) -- (-3,3);
\draw[->,ultra thick] (0,-3) -- (3,-3);
\draw[->,ultra thick] (-3,0) -- (-4.702,-1.702);
\draw[->,ultra thick] (-2,-1) -- (-4,-3);
\draw[->,ultra thick] (-1,-2) -- (-3,-4);
\draw[->,ultra thick] (0,-3) -- (-1.702,-4.702);
\node[my_small_node, label=45:{$4$}] at (0,0){};
\node[my_small_node, label=45:{$4$}] at (2,0){};
\node[my_small_node, label=45:{$4$}] at (0,2){};
\node[my_small_node, label=45:{$4$}] at (-2,-2){};
\node[my_small_node, label=90:{$2$}] at (2,2){};
\node[my_small_node, label=45:{$2$}] at (4,4){};
\node[my_small_node, label=90:{$2$}] at (0,-2){};
\node[my_small_node, label=45:{$2$}] at (0,-4){};
\node[my_small_node, label=90:{$2$}] at (-2,0){};
\node[my_small_node, label=45:{$2$}] at (-4,0){};
\end{tikzpicture}
}
\caption{The $28$ centers of the bitangents to the honeycomb quartic \eqref{eq:f}.} \label{fig:honeycomb}
\end{center}
\end{figure}
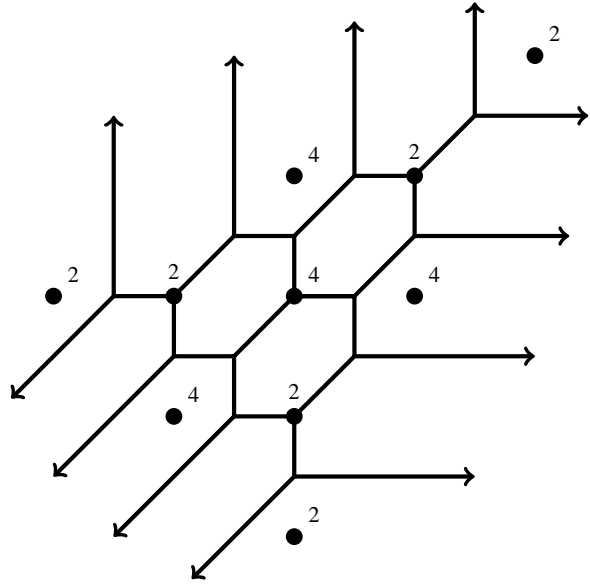

\begin{proposition}\label{p:locations}
The 28 bitangents $Ax+By+z=0$ of $X=V(f)$ include
\begin{itemize}
\item exactly 4 with $(\val(A), \val(B)) = (0, 0)$,
\item exactly 4 with $(\val(A), \val(B)) = (-2, 0)$, 
\item exactly 4 with $(\val(A), \val(B)) = (0, -2),$ 
\item exactly 4 with $(\val(A), \val(B)) = (2, 2),$
\item exactly 2 with $(\val(A), \val(B)) = (-2, -2),$ 
\item exactly 2 with $(\val(A), \val(B)) = (-4, -4),$ 
\item exactly 2 with $(\val(A), \val(B)) = (2, 0),$
\item exactly 2 with $(\val(A), \val(B)) = (4, 0),$
\item exactly 2 with $(\val(A), \val(B)) = (0, 2),$
\item exactly 2 with $(\val(A), \val(B)) = (0, 4).$
\end{itemize}
\end{proposition}

\begin{proof}
The first four statements are immediate from Theorem~\ref{t:hexagonal}.  So by symmetry, we just need to prove that there are exactly two tropicalized bitangents centered at $(2,2)$ and two centered at $(4,4)$.  

Let $(A,B)\in\mathcal{B}$.  First, we use the classical method of the Newton polygon (equivalently, tropical geometry in dimension 1, see e.g.~\cite[Proposition 3.4.8]{ms}) on some special changes of coordinates to get information about the possible values of $A+B$ and $A+B+1$. 
Rewriting the ideal $I\subset K[A,B]$ in terms of coordinates $A'=A+B$ and $B$ and eliminating $B$ from the result produces an ideal  that is principally generated by a polynomial $p(A') $ with the property that 
$$p(A')=0 \qquad \text{if and only if} \qquad A' = A+B \text{ for some } (A,B)\in \mathcal{B}.$$
(The reason that we do not simply eliminate $B$ right away, and recover a polynomial describing all 28 possible values of $A$, is that that computation is actually too large for \emph{Macaulay2} to handle.  In contrast, because of the symmetry of $f$, there are {\em fewer} possible values of $A+B$, and it turns out that this difference is just enough to make the computation of $p$ feasible.)

\begin{claimm} \label{a+b} The polynomial $p$ is a squarefree polynomial of degree 16.  The roots of $p$ have valuations $-4,-2,0,$ and $2$, and these valuations are attained with multiplicity $1,5,7,$ and $3$ respectively.
\end{claimm}
\begin{proof}
We calculated $p$ in \emph{Macaulay2}. To check that it is squarefree, we computed a specialization, say $t=1$, of the resultant of $p$ and $p'$ and noted that it is nonzero (the actual resultant of $p$ and $p'$ is much too large to compute exactly).  The valuations of the roots of $p$ are determined, with multiplicity, by the valuations of the 17 coefficients of $p$, via the method of the Newton polygon.  The conclusion follows.  
\qed \end{proof}

\begin{claimm}\label{a+b+1}
If $(A,B)\in\mathcal{B}$ then $\val(A+B+1)\le 0$.
\end{claimm}
\begin{proof}
Using \emph{Macaulay2}, we rewrote $I$ in terms of coordinates $A''=A+B+1$ and $B$, then eliminated $B$.  The result is a polynomial whose Newton polygon is comprised of segments of nonnegative slope.
\qed \end{proof}

\begin{claimm}\label{a+a}
There are exactly four points of $\mathcal{B}$ of the form $(A,A)$; their valuations are $(0,0)$ and $(2,2)$ with multiplicity 2 each.
\end{claimm}

\begin{proof}
We substituted $B=A$ into $I$ and computed a Gr\"obner basis. The result is an ideal principally generated by a degree four polynomial; analyzing its Newton polygon yields the claim about the valuations.
\qed \end{proof}

\begin{claimm}\label{2or4}
If $(A,B)\in\mathcal{B}$ such that $(\val(A),\val(B)) =(-a,-a)$ for some $a\ge 2$, then $\val(A+B)=-a$.  Hence $(\val(A),\val(B))$ is actually either $(-4,-4)$ or $(-2,-2)$.
\end{claimm}
\begin{proof}
The second statement follows from the first by Claim~\ref{a+b}.  To prove the first, suppose instead $\val(A+B) > \val(A)=\val(B)$.  Now by symmetry, if $Ax+By+z=0$ is a bitangent, then $Ax+y+Bz=0$ is also a bitangent; that is, $(\frac{A}{B},\frac{1}{B})\in\mathcal{B}$.  Therefore, by Claim~\ref{a+b+1}, $\val(\frac{A}{B}+\frac{1}{B}+1) \le 0$.  But $\val(\frac{1}{B}) > 0$ and $\val(\frac{A+B}{B}) > 0$, contradiction.
\qed \end{proof}

Now we have assembled all the computational results we needed to prove Proposition~\ref{p:locations}.  First, by Claim~\ref{a+b}, exactly 16 distinct values of $A+B$ occur. Now by Claim~\ref{a+a}, there are four bitangents of the form $(A,A)$; furthermore, the remaining 24 bitangents form twelve pairs $(A,B)$ and $(B,A)$, by symmetry of $f$.  From these facts we conclude that if two points $(A,B)\ne (A',B')\in\mathcal{B}$ are such that $A+B=A'+B'$, then $(A',B')=(B,A)$.  Then the fact that $p$ has a unique root of valuation $-4$ (Claim~\ref{a+b}) implies that there is exactly one pair of bitangents $(A,B)\ne (B,A)\in\mathcal{B}$ with $\val(A+B) = -4$.  
By Claim~\ref{2or4}, these must be the unique pair of bitangents that have valuation $(-4,-4)$: indeed, all other cases given in Theorem~\ref{t:hexagonal} give pairs $(A,B)$ with $\val(A+B)>-4$.  Therefore, by Claim~\ref{2or4} and Theorem~\ref{t:hexagonal} again, exactly two bitangents have valuation $(-2,-2)$.
\qed \end{proof}

\begin{observation}\label{o:branching}
Suppose $(A,B)\in\mathcal{B}$ is one of $m$ bitangents that have the same tropicalization.  Let $n\in \mathbb{Z}_{>0}$ be the smallest number such that $A,B\in \CC(\!(t^{1/n})\!)$.  Then $n\le m$.  Indeed, the automorphisms $\CC(\!(t^{1/n})\!) / \CC(\!(t)\!)$ preserve valuations, so already produce $n$ bitangents with the same tropicalization.
\end{observation}

Now, fix a point $(a,b)\in\RR^2$ that is a possible tropicalization of $(A,B)\in\mathcal{B}$, as determined in Proposition~\ref{p:locations}.  Fix an integer $n$.  We now explain how to compute, at least in principle, the Puiseux expansions of the bitangents $(A,B)$ that lie in $\CC(\!(t^{1/n})\!)^2$ and whose tropicalizations are $(a,b)$.  After that, we will explain the alteration that we need for the computation to succeed in practice.

First, after a base change $s^n=t$, we may assume $n=1$; then, after a change of coordinates, we may assume $(a,b) = (0,0)$.  Then by \cite[Proposition 4.4]{gubler}, the variety
$$\{(a,b)\in\CC^2~:~ \text{there is } (A,B)=(a+\text{higher terms}, b+ \text{higher terms})\in\mathcal{B}\}$$
is cut out by the ideal $(I\cap R[A,B])|_{t=0}$. 

Then, for each $(a,b)\in(\CC^*)^2$ in the variety above, we make a change of coordinates $$A=a+tA', \qquad B=b+tB'$$ and repeat the process to obtain the next coefficient in each expansion, and so on, to desired precision.  
We repeat this computation for each possible location found in Proposition~\ref{p:locations} and for each $n\le m$ to guarantee, by Observation~\ref{o:branching}, that we find all 28 bitangents.  In our example, only $m=2$ and $m=4$ occur. 

Now, this algorithm works in principle, but in practice, finding the saturation $(I:t^\infty)$ is much too slow.  So when we run the algorithm above, we make a small change: we simply use $I|_{t=0}$, which is easy to compute, instead of $(I:t^\infty)|_{t=0}$.  We have a scheme-theoretic inclusion $V(I|_{t=0}) \supseteq V((I:t^\infty)|_{t=0})$, so in particular, if $\dim V(I|_{t=0}) = 0$ then $\text{deg } V(I|_{t=0}) \ge \text{deg } V((I:t^\infty)|_{t=0})$. 

In other words, using $I|_{t=0}$ at each stage gives us only a set-theoretic upper bound on the possible bitangents $(A,B)\in\mathcal{B}$, computed to arbitrary precision; but if ever the {\em sum} of the degrees of all of the zero-dimensional schemes in question drops to $28$, then we know that the 28 Puiseux expansions we have computed up to that point do correspond to true bitangents.  This termination condition does indeed happen in our example.  The results of our computations are shown in Table~\ref{table:28}.

\begin{table}[t] 
\centering 
\begin{tabular}{c c c} 
\hline\hline 
$\val$ & $(A,B)$ \\ [0.7ex] 
\hline 
\footnotesize $(0,0)$ & \footnotesize$(1,1), (1+4t+4t^3-24t^4+\dots, 1+4t+4t^3-24t^4+\dots)$ \\ [0.7ex]
 & \footnotesize $\left(1, 1-4t+16t^2-68t^3+\dots\right), \left(1-4t+16t^2-68t^3+\dots, 1\right)$ \\ [0.7ex]
\footnotesize $(2,2)$ & \footnotesize $\left(t^2 + 2i \cdot t^{\frac{5}{2}} -2t^3 - 5i \cdot t^{\frac{7}{2}} + \dots, t^2 + 2i t^{\frac{5}{2}} -2t^3 - 5i \cdot t^{\frac{7}{2}} + \dots\right)$ \\ [0.7ex]
 & \footnotesize $\left(t^2 + 2i \cdot t^{\frac{5}{2}} -2t^3-5i \cdot t^{\frac{7}{2}} + \dots, t^2 - 2i t^{\frac{5}{2}} -2t^3+5i\cdot t^{\frac{7}{2}}+ \dots\right)$ \\ [0.7ex]
& \footnotesize $\left(t^2 - 2i \cdot t^{\frac{5}{2}} -2t^3 + 5i \cdot t^{\frac{7}{2}} + \dots, t^2 - 2i t^{\frac{5}{2}} -2t^3 + 5i \cdot t^{\frac{7}{2}} + \dots\right)$ \\ [0.7ex]
 & \footnotesize $\left(t^2 - 2i \cdot t^{\frac{5}{2}} -2t^3+5i \cdot t^{\frac{7}{2}} + \dots, t^2 + 2i t^{\frac{5}{2}} -2t^3-5i\cdot t^{\frac{7}{2}}+ \dots\right)$ \\ [0.7ex]
\footnotesize $(0,-2)$ & \footnotesize $\left(1,t^{-2} + 2i \cdot t^{-\frac{3}{2}} -2t^{-1} - 5i \cdot t^{-\frac{1}{2}} + \dots \right)$ \\ [0.7ex] 
& \footnotesize $\left(1,t^{-2} - 2i \cdot t^{-\frac{3}{2}} -2t^{-1} + 5i \cdot t^{-\frac{1}{2}} + \dots \right)$ \\ [0.7ex]
& \footnotesize $\left(1+4i \cdot t^{\frac{1}{2}} - 8t -18it^{\frac{3}{2}} + \dots, t^{-2} +2i \cdot t^{-\frac{3}{2}} -2t^{-1} -5i \cdot t^{-\frac{1}{2}} + \dots \right)$ \\ [0.7ex]
& \footnotesize $\left(1-4i \cdot t^{\frac{1}{2}} - 8t +18it^{\frac{3}{2}} + \dots, t^{-2} -2i \cdot t^{-\frac{3}{2}} -2t^{-1} +5i \cdot t^{-\frac{1}{2}} + \dots \right)$ \\ [0.7ex]
\footnotesize $(-2,0)$ & \footnotesize $\left(t^{-2} + 2i \cdot t^{-\frac{3}{2}} -2t^{-1} - 5i \cdot t^{-\frac{1}{2}} + \dots, 1 \right)$ \\ [0.7ex]
& \footnotesize $\left(t^{-2} - 2i \cdot t^{-\frac{3}{2}} -2t^{-1} + 5i \cdot t^{-\frac{1}{2}} + \dots, 1 \right)$ \\ [0.7ex]
& \footnotesize $\left(t^{-2} +2i \cdot t^{-\frac{3}{2}} -2t^{-1} -5i \cdot t^{-\frac{1}{2}} + \dots, 1+4i \cdot t^{\frac{1}{2}} - 8t -18it^{\frac{3}{2}} + \dots \right)$ \\ [0.7ex]
& \footnotesize $\left(t^{-2} -2i \cdot t^{-\frac{3}{2}} -2t^{-1} +5i \cdot t^{-\frac{1}{2}} + \dots, 1-4i \cdot t^{\frac{1}{2}} - 8t +18it^{\frac{3}{2}} + \dots \right)$ \\ [0.7ex]
\footnotesize $(2,0)$ & \footnotesize $\left(4t^2+4i\cdot t^{\frac{5}{2}}-12t^3-18i \cdot t^{\frac{7}{2}} + \dots, 1+2i \cdot t^{\frac{1}{2}} -2t -5i \cdot  t^{\frac{3}{2}} + \dots \right)$ \\ [0.7ex]
& \footnotesize $\left(4t^2-4i\cdot t^{\frac{5}{2}}-12t^3+18i \cdot t^{\frac{7}{2}} + \dots, 1-2i \cdot t^{\frac{1}{2}} -2t +5i \cdot t^{\frac{3}{2}} + \dots \right)$ \\ [0.7ex]
\footnotesize $(0,2)$ & \footnotesize $\left(1+2i \cdot t^{\frac{1}{2}} -2t -5i \cdot t^{\frac{3}{2}} + \dots, 4t^2+4i\cdot t^{\frac{5}{2}}-12t^3-18i \cdot t^{\frac{7}{2}} + \dots \right)$ \\ [0.7ex]
& \footnotesize $\left(1-2i \cdot t^{\frac{1}{2}} -2t +5i \cdot t^{\frac{3}{2}} + \dots, 4t^2-4i\cdot t^{\frac{5}{2}}-12t^3+18i \cdot t^{\frac{7}{2}} + \dots \right)$ \\ [0.7ex]
\footnotesize $(-2,-2)$ & \footnotesize $\left(\frac{1}{4} t^{-2} + \frac{i}{4} t^{-\frac{3}{2}} + \frac{1}{2} t^{-1} + \frac{i}{8} t^{-\frac{1}{2}} + \dots, \frac{1}{4}t^{-2} -\frac{i}{4} t^{-\frac{3}{2}} + \frac{1}{2} t^{-1} - \frac{i}{8} t^{-\frac{1}{2}} + \dots \right)$ \\ [0.7ex]
& \footnotesize $\left(\frac{1}{4}t^{-2} -\frac{i}{4} t^{-\frac{3}{2}} + \frac{1}{2} t^{-1} - \frac{i}{8} t^{-\frac{1}{2}} + \dots , \frac{1}{4} t^{-2} + \frac{i}{4} t^{-\frac{3}{2}} + \frac{1}{2} t^{-1} + \frac{i}{8} t^{-\frac{1}{2}} + \dots\right)$ \\ [0.7ex]
\footnotesize $(4,0)$ & \footnotesize $\left(2t^4+2i \cdot t^{\frac{9}{2}} -10t^5 -13i \cdot t^{\frac{11}{2}} + \dots, 1 + 2i \cdot t^{\frac{1}{2}} -2t-5i \cdot t^{\frac{3}{2}} + \dots \right)$ \\ [0.7ex]
& \footnotesize $\left(2t^4-2i \cdot t^{\frac{9}{2}} -10t^5 +13i \cdot t^{\frac{11}{2}} + \dots, 1 - 2i \cdot t^{\frac{1}{2}} -2t+5i \cdot t^{\frac{3}{2}} + \dots \right)$ \\ [0.7ex]
\footnotesize $(0,4)$ & \footnotesize $\left(1 + 2i \cdot t^{\frac{1}{2}} -2t-5i \cdot t^{\frac{3}{2}} + \dots, 2t^4+2i \cdot t^{\frac{9}{2}} -10t^5 -13i \cdot t^{\frac{11}{2}} + \dots \right)$ \\ [0.7ex]
& \footnotesize $\left(1 - 2i \cdot t^{\frac{1}{2}} -2t+5i \cdot t^{\frac{3}{2}} + \dots, 2t^4-2i \cdot t^{\frac{9}{2}} -10t^5 +13i \cdot t^{\frac{11}{2}} + \dots \right)$ \\ [0.7ex]
\footnotesize $(-4,-4)$ & \tiny $\left(\frac{1}{2} t^{-4} + \frac{i}{2} t^{-\frac{7}{2}} +2t^{-3} +\frac{5}{4} i \cdot t^{-\frac{5}{2}} + \dots, \frac{1}{2} t^{-4} - \frac{i}{2} t^{-\frac{7}{2}} +2t^{-3} -\frac{5}{4} i \cdot t^{-\frac{5}{2}} + \dots \right)$ \\  [0.7ex]
& \tiny $\left(\frac{1}{2} t^{-4} - \frac{i}{2} t^{-\frac{7}{2}} +2t^{-3} -\frac{5}{4} i \cdot t^{-\frac{5}{2}} + \dots, \frac{1}{2} t^{-4} + \frac{i}{2} t^{-\frac{7}{2}} +2t^{-3} +\frac{5}{4} i \cdot t^{-\frac{5}{2}} + \dots \right)$ \\
[1ex] 
\hline 
\end{tabular} 
\caption{Puiseux Series Expansion of $(A,B) \in \mathcal{B}$.} 
\label{table:28}
\end{table}

\bigskip

\begin{acknowledgement}
We are grateful to M.~Baker, Y.~Len, R.~Morrison, N.~Pflueger, and Q.~Ren for generously sharing their ideas on tropical plane quartics developed in the paper \cite{bitangent}.  Thanks also to M.~Baker, Y.~Len, and B.~Sturmfels for helpful comments on an earlier version of this paper, and J.~Rabinoff for helpful references.  We heartily thank M.~Manjunath, M.~Panizzut, and two anonymous referees for extensive and insightful comments on a previous version of this paper. We also thank W.~Stein and SageMathCloud for providing computational resources.  MC was supported by NSF DMS-1204278.  PJ was supported by the Harvard College Research Program during the summer of 2014.
\end{acknowledgement}

\end{document}